\documentclass{article}
\RequirePackage{amsthm,amsmath,amsfonts,amssymb}
\RequirePackage[numbers]{natbib}
\RequirePackage{tikz}
\RequirePackage{hyperref}
\hypersetup{
    colorlinks,      
    urlcolor=blue,
    citecolor=blue,
}
\usepackage{calrsfs}
\usepackage{arydshln}
\usepackage{blkarray}
\usepackage{algorithm,algpseudocode}

\RequirePackage{graphicx}
\usepackage{nicematrix}

\usepackage[margin=0.7in]{geometry}

\usepackage{booktabs}

\usepackage{subcaption}
\usepackage{fullpage,graphicx,psfrag,amsmath,amsfonts,verbatim,enumitem,amsthm,listings,pdfpages}
\usepackage{listings}
\usepackage{xcolor}

\definecolor{codegreen}{rgb}{0,0.6,0}
\definecolor{codegray}{rgb}{0.5,0.5,0.5}
\definecolor{codepurple}{rgb}{0.58,0,0.82}
\definecolor{backcolour}{rgb}{0.95,0.95,0.92}

\lstdefinestyle{mystyle}{
    backgroundcolor=\color{backcolour},   
    commentstyle=\color{codegreen},
    keywordstyle=\color{magenta},
    numberstyle=\tiny\color{codegray},
    stringstyle=\color{codepurple},
    basicstyle=\ttfamily\footnotesize,
    breakatwhitespace=false,         
    breaklines=true,                 
    captionpos=b,                    
    keepspaces=true,                 
    numbers=left,                    
    numbersep=5pt,                  
    showspaces=false,                
    showstringspaces=false,
    showtabs=false,                  
    tabsize=2
}
\newcommand{\ones}{\mathbf 1}
\newcommand{\reals}{\mathbb{R}}

\newcommand{\naturals}{\mathbb{N}}
\newcommand{\id}{\mathrm{Id}}

\newcommand{\X}{{\mathcal{X}}}

\newcommand{\Prob}{\mathbb{P}}
\usepackage{pgfplots}
\pgfplotsset{compat=1.15}
\usepackage{mathrsfs}

\usetikzlibrary{arrows}

\theoremstyle{plain}

\newtheorem{theorem}{Theorem}[section]
\newtheorem{lemma}[theorem]{Lemma}
\newtheorem{proposition}[theorem]{Proposition}

\newtheorem{conjecture}[theorem]{Conjecture}
\newtheorem{question}[theorem]{Question}

\theoremstyle{definition}
\newtheorem{definition}[theorem]{Definition}

\newtheorem{remark}[theorem]{Remark}
\newtheorem{remark*}{Remark}

\title{Limit profiles and cutoff for the Burnside process on Sylow double cosets}
\author{Michael Howes}

\begin{document}
\maketitle
\begin{abstract}
    This article gives sharp estimates for the mixing time of the Burnside process for Sylow \(p\)-double cosets in the symmetric group \(S_n\). This process is a Markov chain on \(S_n\) which can be used to uniformly sample Sylow \(p\)-double cosets. The analysis applies when \(n = pk\) with \(p\) prime and \(k < p\). The main result describes the limit profile of the distance to the stationary distribution as \(p\) goes to infinity. From the limit profile, we get the following two corollaries. First, if \(k\) remains fixed as \(p \to \infty\), then order \(p\) steps are necessary and sufficient for mixing and cutoff does not occur. Second, if \(k \to \infty\) as \(p \to \infty\), then cutoff occurs at \(p \log k\) with a window of size \(p\). The limit profile is derived from explicit upper and lower bounds on the distance between the Burnside process and its stationary distribution. These non-asymptotic bounds give very accurate approximations even for \(p\) as small as \(11\).
\end{abstract}

\section{Introduction and main results}

The Burnside process is a Markov chain Monte Carlo algorithm for uniformly sampling orbits from a group action~\cite{jerrum1993uniform}. Recent works have shown empirically that the Burnside process mixes quickly in a number of complex and interesting examples \cite{bartholdi2024algorithm,diaconis2025random,diaconis2025counting}. This paper is perhaps the first to provide a sharp analysis of the mixing time of an instance of the Burnside process for double cosets.

To describe the Burnside process, let \(\mathcal{X}\) be a finite set and let \(G\) be a finite group acting on \(\mathcal{X}\) with \(x^g\) denoting the image of \(g\) acting on \(x\). The set \(\X\) splits into disjoint orbits
\[
    \X = \mathcal{O}_{x_1} \sqcup \mathcal{O}_{x_2} \sqcup \cdots \sqcup \mathcal{O}_{x_Z},
\]
where \(\mathcal{O}_x = \{x^g : g \in G\}\) is the orbit containing \(x\) and \(Z\) is the number of orbits. The Burnside process is a Markov chain on \(\X\) with stationary distribution \(\pi(x) = \frac{1}{Z|\mathcal{O}_x|}\). Thus, sampling \(X\) from \(\pi\) and reporting the orbit containing \(X\) gives a way of uniformly sampling an orbit. Furthermore, taking many steps of the Burnside produces random elements of \(\X\) that are approximately distributed according to \(\pi\).

The Burnside process transitions from \(x\) to \(y\) according to the following two-step procedure:
\begin{enumerate}
    \item From \(x \in \mathcal{X}\), choose \(g\) uniformly from the set \(G_x = \{g \in G : x^g = x\}\).
    \item Given \(g\), choose \(y\) uniformly from the set \(\X^g = \{y \in \X : y^g = y\}\).
\end{enumerate}
The probability of transitioning from \(x\) to \(y\) is thus given by the following transition kernel
\begin{equation}\label{eq:kernel}
    P(x,y) = \frac{1}{|G_x|} \sum_{g \in G_x \cap G_y} \frac{1}{|\X^g|},
\end{equation}
The kernel \(P\) is ergodic and reversible with respect to \(\pi\) and thus \(\pi\) is the unique stationary distribution of \(P\)~\cite[Theorem~2.1]{jerrum1993uniform}. For \(x \in \X\) and \(t \ge 1\) we will write \(P^t_x\) for the distribution on \(\X\) that corresponds to starting at \(x\) and taking \(t\) steps of the Burnside process. That is for any \(y \in \X\),
\[
    P_x^1(y) =P_x(y)= P(x,y) \quad \text{and} \quad P_x^t(y) = \sum_{z \in \X} P^{t-1}_x(z)P_z(y).
\]
This paper focuses on a particular case of the Burnside process called the \emph{Sylow--Burnside process}. In this example, \(\X = S_n\) is the symmetric group of degree \(n\) and \(G\) is the product group \(H \times H\) where \(H \subseteq S_n\) is a Sylow \(p\)-subgroup with \(p\) prime. The product \(H \times H\) acts on \(S_{n}\) by
\[
    \sigma^{h,g}=h^{-1}\sigma g.
\]
Under this group action the orbits are Sylow \(p\)-double cosets and the Sylow--Burnside process can be used to uniformly sample double cosets. Background and motivation for studying this example is given in Section~\ref{sec:sylow-background}.

If \(P\) is an ergodic transition kernel, then the distribution \(P_x^t\) converges to \(\pi\) as \(t\) goes to infinity for any \(x \in \X\). To understand how quickly \(P_x^t\) converges to \(\pi\),  \emph{total variation distance} is typically used. The total variation distance between two probability measures \(\mu\) and \(\nu\) on a finite set \(\X\) is given by
\[
    \Vert \mu - \nu \Vert_{\mathrm{TV}} = \sup_{A \subseteq \X}|\mu(A) - \nu(A)|=\frac{1}{2}\sum_{y \in \X}|\mu(x)-\nu(x)| = \sum_{y \in \X}(\mu(x)-\nu(x))_+.
\]
The field of Markov chain mixing is focused on proving upper bounds on \(\Vert P^t_x - \pi \Vert_{\mathrm{TV}}\) that depend on both \(t\) and parameters of the Markov chain \(P\) (such as the size of \(\X\)). The worst-case value of \(\Vert P^t_x - \pi \Vert_{\mathrm{TV}}\) (as a function of \(x\)) is denoted by \(d(t)\). That is, for \(t \in \naturals\),
\begin{align}\label{eq:dt}
    d(t) & = \max_{x \in \X}\Vert P^t_x - \pi \Vert_{\mathrm{TV}},
\end{align}
and \(d(t)=d(\lfloor t \rfloor)\) for noninteger \(t\) where \(\lfloor \cdot \rfloor\) is the floor function.

The goal of this paper is to approximate \(d(t)\) when \(P\) is the transition kernel for the Sylow--Burnside process. The analysis in this paper applies when \(n =pk\) with \(p\) prime and \(k < p\). The main result is the limit of \(d(t)\) as \(p,t \to \infty\) jointly. A result of this type is called a \emph{limit profile} or a \emph{cutoff profile}.
\begin{theorem}\label{thrm:mix}
    Suppose \(p=p_m\) is a diverging sequence of prime numbers, and that \(k=k_m\) is a sequence of positive integers with \(k_m < p_m\) for all \(m\). Let \(P\) be the transition kernel for the Sylow--Burnside process on \(S_{pk}\) and let \(d(t)\) be as in equation \eqref{eq:dt}.
    \begin{enumerate}
        \item If \(k_m \equiv k\) is eventually constant, then for all \(c \ge 0\),
              \begin{equation}\label{eq:fixed-k}
                  d(cp)\rightarrow 1-(1-e^{-c})^k.
              \end{equation}
        \item If \(k_m \to \infty\), then for all \(c \in \reals\),
              \begin{equation}\label{eq:cutoff}
                  d(p\log k + cp)  \to 1-\exp({-e^{-c}}). 
              \end{equation}
    \end{enumerate}
\end{theorem}
\begin{remark}
    \begin{enumerate}
        \item Theorem~\ref{thrm:mix} describes how fast \(t\) must grow as a function of \(p\) for \(d(t)\) to be small. If \(k\) is fixed, then taking \(t=cp\) with \(c\) large is necessary and sufficient for \(d(t)\) to be small. If \(k\) diverges with \(p\), then taking \(t = p\log k + cp\) with \(c\) large is necessary and sufficient for \(d(t)\) to be small.
        \item Equation \eqref{eq:cutoff} is an example of \emph{cutoff} \cite{diaconis1996cutoff}. If \(t = (1+\varepsilon)p\log k\), then \(d(t)\) approaches 0 and if \(t=(1-\varepsilon)p\log k\), then \(d(t)\) approaches 1. Demonstrating that a Markov chain has cutoff often requires detailed understanding of the Markov chain and its asymptotics.
        \item Theorem~\ref{thrm:mix} gives the limiting value of \(\Vert P^t_\sigma - \pi \Vert_{\mathrm{TV}}\) for the worst-case choice of \(\sigma \in S_{n}\). In the proof of Theorem~\ref{thrm:mix} we will see that this worst-case corresponds to setting \(\sigma =\id\) or any other \(\sigma\) with \(|H\sigma H|= |H|\). For general \(\sigma\), see the remarks after Theorem~\ref{thrm:bound}.
    \end{enumerate}
\end{remark}

The asymptotic results in Theorem~\ref{thrm:mix} are a consequence of the following bounds on \(d(t)\). Unlike Theorem~\ref{thrm:mix}, Theorem~\ref{thrm:bound} contains an explicit error bound that decays super-exponentially as \(p \to \infty\).
\begin{theorem}\label{thrm:bound}
    Fix \(p \ge 11\), \(1 \le k <p\) and \(\sigma \in S_{pk}\). If \(|H\sigma H|=p^a\), then for all \(t \ge 1\)
    \begin{equation}\label{eq:bounds}
        \left|  \Vert P^t_\sigma - \pi\Vert_{\mathrm{TV}} - 1+\left(1-\left(1-\frac{1}{p}\right)^t\right)^{2k-a}\right| \le \frac{4p^4}{(p-1)!} + \frac{t}{(p-2)!}.
    \end{equation}
    In particular,
    \begin{equation}\label{eq:upper-bound}
        \Vert P^t_\sigma - \pi \Vert_{\mathrm{TV}}  \le 1-\left(1-\left(1-\frac{1}{p}\right)^t\right)^{2k-a}+\frac{4p^4}{(p-1)!} + \frac{t}{(p-2)!}.
    \end{equation}
\end{theorem}
\begin{remark}
    \begin{enumerate}
        \item  Somewhat surprisingly, the right-hand-side of \eqref{eq:upper-bound} diverges as \(t \to \infty\) for fixed \(p\). However, the function \(d(t)\) is non-decreasing and satisfies \(d(\ell t) \le (2d(t))^\ell\) \cite[Chapter~4]{Levin}. Thus, if \eqref{eq:upper-bound} is used to find \(t_0\) such that \(d(t_0)\le 1/4\), then for all \(t \ge t_0\) we will have \(d(t) \le 2^{-\lfloor t/t_0 \rfloor}\). It follows that \eqref{eq:upper-bound} can be adapted to give bounds on \(d(t)\) that go to zero even with fixed \(p\). Note that Theorem~\ref{thrm:mix} implies that the smallest such \(t_0\) is asymptotic to \(p\log k\) as \(p,k \to \infty\).
        \item Furthermore, if \(p\) is even moderately large, then the term \(\frac{t}{(p-2)!}\) is negligible. For instance if \(p=11\), then \(\frac{4p^4}{(p-1)!} < 0.017\) and \(\frac{1}{(p-2)!} < 3\times 10^{-6}\). Thus, \eqref{eq:bounds} gives the value of \(d(t)\) up to an error of less than \(0.02\) for all \(t \le 1000\). Simulations in Section~\ref{sec:sim} also show the accuracy of \eqref{eq:bounds}.
        \item From \eqref{eq:bounds} it is possible to derive the limit profile of \(\Vert P^t_\sigma - \pi \Vert_{\mathrm{TV}}\) for any sequence of permutations \(\sigma \in S_{pk}\) with \(p=p_m \to \infty\). The limit profile is the same as in Theorem~\ref{thrm:mix} except \(k\) is replaced with \(2k-a\) where \(|H\sigma H|=p^a\) as above. In particular, if \(a=2k\), then just a single step of the Burnside process is sufficient as \(p \to \infty\).
        \item The condition \(p \ge 11\) and the error terms in \eqref{eq:bounds} come from lower bounds on the number of double cosets of maximal size. These bounds have not been optimized and the condition \(p \ge 11\) can be weakened. However, the bounds are more than sharp enough to give the asymptotic result in Theorem~\ref{thrm:mix}.
    \end{enumerate}
\end{remark}

The remainder of this section gives an overview of the two main ideas behind Theorems~\ref{thrm:mix} and \ref{thrm:bound}. Section~\ref{sec:background} contains background on the Burnside process and Sylow \(p\)-double cosets. Section~\ref{sec:sylow--burnside} describes how to implement the Sylow--Burnside process and proves an important lumping property of the Sylow--Burnside process. The proofs of the main results are in Section~\ref{sec:proofs}. Section~\ref{sec:examples} contains examples and simulation results and Section~\ref{sec:related} discusses related Markov chains. The article ends in Section~\ref{sec:non-abelian} with conjectures and open questions about extending the results to \(n \ge p^2\).

\subsection{Proof overview}\label{sec:proof-overview}

Theorems~\ref{thrm:mix} and \ref{thrm:bound} rely on two main ideas. The first idea is that the Sylow--Burnside process on \(S_{pk}\) lumps to the size of the double coset. The second idea is that most double cosets are of maximal size.

\subsubsection{Lumping to double coset size}\label{sec:lump-overview}
A Markov chain lumps under a function \(T\) if the following holds:
\begin{definition}
    Let \(P\) be a transition kernel on a set \(\X\) and let \(T:\X \to \mathcal{Y}\) be a function. The transition kernel \(P\) \emph{lumps under \(T\)} if for all Markov chains \((X_t)_{t \ge 0}\) generated by \(P\), the process \(Y_t = T(X_t)\) is a Markov chain on \(\mathcal{Y}\). The process \((Y_t)_{t \ge 0}\) is called a \emph{lumped process}.
\end{definition}
If \(P\) lumps under \(T\), then we will let \(\overline{P}\) denote the transition kernel for the lumped process. The kernel \(\overline{P}\) is called the \emph{lumped kernel}. Furthermore, if \(P\) has a unique stationary distribution \(\pi\), then \(\overline{P}\) has a unique stationary distribution \(\overline{\pi}\). The distribution \(\overline{\pi}\) is called the \emph{lumped stationary distribution} and \(\overline{\pi}(y) = \pi(\{x:T(x)=y\})\).

Lumping a Markov chain can help analyze the mixing time of a Markov chain. This is because the lumped kernel \(\overline{P}\) has a smaller state space and is often easier to analyze. In the current example, the Sylow \(p\)-double cosets in \(S_{pk}\) can have size \(p^k,p^{k+1},\ldots,p^{2k}\). This means that mapping a permutation to the size of its double coset reduces the problem from a state space of size \((pk)!\) to a state space of size \(k+1\). The following result gives an exact formula for the lumped transition kernel of the Sylow--Burnside process. 
\begin{theorem}\label{thrm:lump}
    Define \(T: S_{pk} \to \{k,\ldots 2k\}\) by
    \begin{equation}\label{eq:T}
        T(\sigma)=a \Longleftrightarrow |H\sigma H|=p^a. 
    \end{equation}
    Then the Sylow--Burnside kernel \(P\) lumps under \(T\) and the lumped transition kernel is given by
    \begin{equation}\label{eq:lump-P}
        \overline{P}(a,b)=\sum_{y=0}^{2k-\max\{a,b\}}\binom{2k-a}{y} \left(\frac{p-1}{p}\right)^y  \frac{p^{a+b-2k}}{(p(k-y))!}f(b-y;k-y),
    \end{equation}
    where \(f(a;k)\) is the number of Sylow \(p\)-double cosets in \(S_{pk}\) of size \(p^a\). Furthermore, if \(Z = \sum_{a=k}^{2k}f(a;k)\) is the total number of Sylow \(p\)-double cosets, then the lumped stationary distribution is given by
    \begin{equation}\label{eq:lump-pi}
        \overline{\pi}(a) = \frac{f(a;k)}{Z}.
    \end{equation}
\end{theorem}

Some aspects of a Markov chain can be ``lost'' when going from \(P\) to \(\overline{P}\). For instance, it is always true that if \(P\) is a transition kernel on \(\X\) that lumps under \(T\), then for all \(x \in \X\)
\[
    \Vert \overline{P}^t_{T(x)} - \overline{\pi} \Vert_{\mathrm{TV}} \le \Vert P^t_x - \pi \Vert_{\mathrm{TV}}.
\]
However, the reverse inequality does not always hold---for example it fails when \(T\) is a constant. For the Sylow--Burnside process with \(T\) defined as in \eqref{eq:T}, Lemma~\ref{lem:tv-comparison} states that reverse inequality holds up to a small error that goes to zero as \(p \to \infty\). Thus, for the purposes of studying \(d(t)\), we can replace \(P\) with \(\overline{P}\) and reduce our problem to a much smaller state space.

Lumping has played an important role in the study and implementation of other examples of the Burnside process. \cite{diaconis2005analysis} showed that the Burnside process always lumps to the orbit---meaning the Markov chain lumps under \(T(x)=\mathcal{O}_x\). In \cite{diaconis2005analysis} an explicit formula for the lumped transition kernel was used to bound the mixing time of the Burnside process for \(S_n\) acting on binary strings. \cite{diaconis2020hahn} also study this example and compute the eigenvectors and eigenvalues of the lumped process. More recently, \cite{diaconis2025curiously} compute the eigenvectors and eigenvalues of the unlumped process and show that the mixing time depends heavily on the starting state. In a different direction, \cite{diaconis2025random} study the Burnside process for contingency tables and integer partitions and show that the lumped process can be enormously more efficient to run than the original process. The present example shows that sometimes it is useful to lump the Burnside process even further.

\subsubsection{Most double cosets are large}\label{sec:most-large}

To analyze the lumped Sylow--Burnside process, we need to understand \(f(a;k)\)---the number of Sylow \(p\)-double cosets in \(S_{pk}\) of size \(p^a\). As claimed in Theorem~\ref{thrm:lump}, the stationary distribution of the lumped process is given by
\[
    \overline{\pi}(a) = \frac{f(a;k)}{Z}.
\]
\cite{diaconis2025number} give a formula for \(f(a;k)\) that we restate in \eqref{eq:numcosets}. \cite{diaconis2025number} also give strong bounds on \(f(a;k)\). These bounds show that most Sylow \(p\)-double cosets have size \(p^{2k}\) in the sense that \(f(2k;k)/Z \to 1\) as \(p \to \infty\). This means that the lumped stationary distribution \(\overline{\pi}\) is well approximated by a point mass at \(2k\). Using this approximation, we get that
\[
    \Vert \overline{P}^t_a - \overline{\pi}\Vert_{\mathrm{TV}} \approx 1-\overline{P}^t_{a}(2k).
\]
That is, the total variation distance between \(\overline{P}^t_a\) and \(\pi\) is roughly equal to the probability that \(\overline{P}^t_a\) assigns to \(\{k,\ldots,2k-1\}\). To compute \(\overline{P}^t_a(2k)\) we use another approximation that replaces \(\overline{P}\) with a second transition kernel \(Q\). The approximating transition kernel \(Q\) has a nice probabilistic description. To move from \(a\) to \(b\) under \(Q\) do the following:
\begin{enumerate}
    \item Flip \(2k-a\) independent coins each with heads probability \(1/p\).
    \item Set \(b = a+X\) where \(X\) is the number of coins that land on heads.
\end{enumerate}
From this description, it follows that \(Q_a^t(2k)\) is equal to the probability that the maximum of \(2k-a\) independent geometric random variables is at most \(t\). Thus,
\[
    \Vert \overline{P}^t_a - \overline{\pi}\Vert_{\mathrm{TV}} \approx 1-Q^t_{a}(2k) = 1-\left(1-\left(1-\frac{1}{p}\right)^{t}\right)^{2k-a}.
\]
Proving the above formula for \(Q^t_a(2k)\) and bounding the distance between \(\overline{P}^t_a\) and \(Q^t_a\) gives Theorem~\ref{thrm:bound} and in turn Theorem~\ref{thrm:mix}.

\subsection*{Acknowledgements}

The author would like to thank Persi Diaconis for introducing him to the Burnside process and for motivating the study of the Sylow--Burnside process. Thanks is also given to Logan Bell, Andrew Liu and Timothy Sudijuno for their helpful comments and suggestions.

\section{Background}\label{sec:background}

\subsection{The Burnside process}

This section reviews some important properties of the Burnside process and surveys recent work on other instances of the Burnside process. Let \(G\) be a finite group acting on a finite set \(\X\). As described in the introduction, the Burnside process is a Markov chain on \(\X\) with transitions from \(x\) to \(y\) according to the following two-step rule:
\begin{enumerate}
    \item From \(x \in \mathcal{X}\), choose \(g\) uniformly from the set \(G_x = \{g \in G : x^g = x\}\).
    \item Given \(g\), choose \(y\) uniformly from the set \(\X^g = \{y \in \X : y^g = y\}\).
\end{enumerate}
The Burnside process can equivalently be described as a random walk on a large bipartite graph. The bi-partition is \(\X \sqcup G\) and an edge connects \(x\) to \(g\) if and only if \(x^g=x\). The two steps of the Burnside process are equivalent to picking \(g\) uniformly from the neighbors of \(x\) and then picking \(y\) uniformly from the neighbors of \(g\). Since the identity group element is connected to every \(x \in \X\), this two-step walk is ergodic on \(\X\). Furthermore, the stationary distribution of the random walk is given by:
\begin{equation}\label{eq:stationary}
    \pi(x) = \frac{\deg(x)}{|\{(x,g): x^g=x\}|} = \frac{|G_x|}{|\{(x,g): x^g=x\}|} =\frac{|G|}{|\mathcal{O}_x||\{(x,g): x^g=x\}|}= \frac{1}{Z|\mathcal{O}_x|}.
\end{equation}
In \eqref{eq:stationary}, the last equality holds because \(|\{(x,g) : x^g =x\}|=Z|G|\) by the lemma that is not Burnside's \cite{neumann1979lemma}. The second last equality uses the orbit--stabilizer theorem (for all \(x \in \X\), \(|G|=|\mathcal{O}_x||G_x|\)).

The definition of the Burnside process and the above results were introduced in \cite{jerrum1993uniform} as an approach to computational P\'olya theory. \cite{jerrum1993uniform} proposed using the Burnside process to estimate the number of orbits \(Z\) --- a task that has applications in chemistry, statistics and computer science. To make the Burnside process a tractable algorithm for estimating \(Z\), two things must be satisfied. First, it must be possible to draw uniform samples from both \(G_x\) and \(\X^g\) as in the definition of the Burnside process. Second, it is necessary to have a bound on the mixing time of the Burnside process. Starting with \cite{jerrum1993uniform}, different works have answered these questions for various group actions and found both positive and negative results.

In the negative direction \cite{goldberg2002burnside} showed that there exist group actions where the Burnside process mixes slowly. They construct their examples by connecting the Burnside process to the Swendsen--Wang algorithm for sampling from Pott's models. Other works have studied the mixing time of the Burnside process for different group actions and found more positive results. These examples include Bose--Einstein statistics \cite{diaconis2005analysis, diaconis2020hahn,diaconis2025curiously}, large sparse contingency tables \cite{dittmer2019counting}, set partitions \cite{paguyo2022mixing}, and conjugacy classes in CA groups \cite{rahmani2022mixing}. 

There also exist a number of works that focus on the implementation of the Burnside process for different applications. These include \cite{bartholdi2024algorithm} for sampling P\'olya trees, \cite{diaconis2025random} for sampling contingency tables and integer partitions and \cite{diaconis2025counting} for counting conjugacy classes in uni-triangular matrix groups. These papers derive implementations of the Burnside process and empirically find that the Burnside process mixes quickly.

Beyond the above examples, properly understanding the mixing time of the Burnside process is an open problem. For instance, Theorem~\ref{thrm:mix} is the first time a limit profile has been computed for a natural instance Burnside process. We hope that our analysis can inform the study of the Burnside process for other group actions.

\subsection{Sylow \(p\)-double cosets in the symmetric group}\label{sec:sylow-background}

In this section we describe the Sylow \(p\)-double cosets in the symmetric group and state some results from \cite{diaconis2025number}. Let \(S_n\) be the symmetric group of permutation of size \(n\). A Sylow \(p\)-subgroup of \(S_n\) is a subgroup \(H \subseteq S_n\) of order \(p^m\) where \(p^m\) is the largest power of \(p\) that divides \(|S_n|=n!\). The value of \(m\) is given by
\[
    m = \sum_{a=1}^\infty \left\lfloor \frac{n}{p^a}\right\rfloor.
\]
The structure of the subgroup \(H\) was determined by \cite{kaloujnine1948structure} and can be described as the automorphism group of a \(p\)-ary tree \cite{wildon2016sylow}. The product group \(H \times H\) acts on the symmetric group by
\[
    \sigma^{h,g} = h^{-1}\sigma g.
\]
The orbits under this group action are the \emph{Sylow \(p\)-double cosets}. A Sylow \(p\)-double coset is a set of the form
\[
    H\sigma H = \{h\sigma g : h,g \in H\}.
\]
The Sylow \(p\)-double cosets form a partition of \(S_n\) and the collection of all double cosets is represented as \(H \backslash S_n/H\). By the orbit--stabilizer theorem, the size of \(H \sigma H\) is given by
\[
    |H \sigma H| = \frac{|H|^2}{|H \cap \sigma^{-1} H \sigma|}.
\]
The subgroup \(H\) has size \(p^m\) and, by Lagrange's theorem \(|H \cap \sigma^{-1} H \sigma|\) must divide \(p^m\). It follows that the size of \(|H\sigma H|\) must be one of \(p^m,p^{m+1},\ldots,p^{2m}\). \cite{diaconis2025number} show that, except for a handful of values of \(n\) and \(p\), each possible size occurs as \(\sigma\) ranges over \(S_n\). They also show that, as long as \(p \ge 3\),  the proportion of permutations \(\sigma\) in double coset of size \(p^{2m}\) goes to \(1\) as \(n \to \infty\). Thus, most permutations are in large double cosets.

The results in this paper are restricted to the ``Abelian case'' when \(n=pk\) with \(k < p\) and so \(m=k\). The name refers to the fact that the Sylow p-subgroup \(H \subseteq S_n\) is Abelian if and only if \(n < p^2\). When \(n=pk\) with \(k<p\), the Sylow \(p\)-subgroup \(H\) is isomorphic to \(C_p^k\)--the product of \(k\) cyclic groups of order \(p\). 

As in the general case, the possible values for \(|H\sigma H|\) are \(p^k,p^{k+1},\ldots, p^{2k}\). \cite{diaconis2025number} compute the number of Sylow \(p\)-double cosets of each of these possible sizes.
\begin{theorem}[Theorem~3.1 in \cite{diaconis2025number}]\label{thrm:formula}
    Let \(f(a;k)\) be the number of Sylow \(p\)-double cosets of size \(p^a\) in \(S_{pk}\). Then for all \(k<p\) and \(a\in \{k,\ldots,2k\}\)
    \begin{equation}
        \label{eq:numcosets}
        f(a;k) = \frac{1}{p^a}\sum_{j=2k-a}^{k}(-1)^{j-(2k-a)}((k-j)p)!j!\binom{k}{j}^2(p(p-1))^j\binom{j}{2k-a}.
    \end{equation}
\end{theorem}
\cite{diaconis2025number} also give the follow bounds on \(f(2k;k)\).
\begin{theorem}[Theorem~3.3 in \cite{diaconis2025number}]\label{thrm:most-large}
    For all \(p\) prime and \(k< p\),
    \[
        \frac{(kp)!}{p^{2k}}\left(1-\frac{1}{(p-2)!}\right) \le f(2k;k) \le \frac{(kp)!}{p^{2k}}.
    \]
\end{theorem}

The following result is not stated in \cite{diaconis2025number} but can be derived from their proof of Theorem~3.3 and the remarks that follow. For completeness, a proof is given in Section~\ref{sec:additional-proofs} of the appendix.
\begin{proposition}\label{prop:most-large}
    Let \(Z = \sum_{a=k}^{2k} f(a;k)\) be the total number of Sylow \(p\)-double cosets. Then for all \(p \ge 11\) prime and \(k<p\),
    \[
        Z\left(1-\frac{2p^4}{(p-1)!}\right) \le f(2k;k) \le Z.
    \]
\end{proposition}
Both Theorem~\ref{thrm:most-large} and Proposition~\ref{prop:most-large} are quantitative versions of the statement ``most Sylow \(p\)-double cosets are of size \(p^{2k}\).'' However, the two results correspond to different meanings of ``most.'' Theorem~\ref{thrm:most-large} states that if \(\sigma\) is chosen uniformly from \(S_{pk}\), then the probability that \(\sigma\) is in a large double-coset goes to 1 as \(p\) goes to infinity. That is, \(p^{2k}f(2k;k)/(pk)! \to 1\). On the other hand, Proposition~\ref{prop:most-large} states that if the double coset is chosen uniformly at random, then the probability that double coset is of size \(p^{2k}\) goes to \(1\). That is, \(f(2k;k)/Z \to 1\). Both these results hold in a very strong sense. They are uniform in \(k<p\) and the error terms go to zero super-exponentially. Theorem~\ref{thrm:most-large} and Proposition~\ref{prop:most-large} are both used to prove the mixing time result in Theorem~\ref{thrm:bound}.

There are many other things to say about Sylow \(p\)-subgroups and double cosets in \(S_n\). See \cite{wildon2016sylow} and \cite{diaconis2025number} for a growing literature.

\section{The Sylow--Burnside process}\label{sec:sylow--burnside}

In this section we study the Sylow Burnside process. Section~\ref{sec:running} contains an implementation of the Sylow--Burnside process and Section~\ref{sec:lump-proof} proves that the Sylow--Burnside process lumps to double coset size. We begin by setting some notation and by fixing a particular Sylow \(p\)-subgroup to work with.

For integers \(a \le b\), \([a:b]\) will be the set \(\{a,a+1,\ldots,b\}\) of size \(b-a+1\). Through-out \(H\) will be a fixed Sylow \(p\)-subgroup in \(S_{pk}\) with \(p\) prime and \(1 \le k < p\). Without loss of generality, we will take \(H\) to be the subgroup
\begin{equation}
    \label{eq:H-def}
    H = \langle \eta_1,\ldots,\eta_k \rangle,
\end{equation}
where \(\eta_j\) is the cyclic permutation \(((j-1)p+1,\ldots,jp)\). The permutations \(\eta_j\) have order \(p\) and commute. This means that each \(g \in H\) can be written uniquely as
\begin{equation}\label{eq:H-form}
    g = \eta_1^{i_1}\cdots \eta_k^{i_k} \quad i_j \in [0:p-1].
\end{equation}
It follows that if \(|\{j \in [1:k] : i_j \neq 0\}|=y\), then \(g\) has \(y\) \(p\)-cycles and \((k-y)p\) fixed points.  

\subsection{Implementation of the Sylow--Burnside process}\label{sec:running}

Let \(H\) be the Sylow \(p\)-subgroup defined by \eqref{eq:H-def}. The goal of this section is to describe how to implement the Burnside process for \(H \times H\) acting on \(S_{pk}\) by \(\sigma^{h,g}=h^{-1}\sigma g\). That is, how to draw uniform samples in the following two steps:
\begin{enumerate}
    \item Given \(\sigma \in S_{pk}\), sample \(h,g\) uniformly from the stabilizer \((H \times H)_\sigma=\{(h,g):\sigma^{h,g}=\sigma\}\).
    \item Given \(h,g\), sample \(\tau\) uniformly from the set of fixed points \(S_{pk}^{h,g} = \{\tau : \tau^{h,g}=\tau\}\).
\end{enumerate}
These two steps are described formally in Algorithms~\ref{alg:stablizer} and \ref{alg:fixed-points} respectively and their derivations are given in the next two subsections. Both algorithms are based on results from Section~3 of \cite{diaconis2025random} which gives a general implementation of the Burnside process for double cosets. The key observation is the equivalence
\begin{equation}\label{eq:equiv}
    \sigma^{h,g}=\sigma \Longleftrightarrow h = \sigma g \sigma^{-1}.
\end{equation}

Algorithms~\ref{alg:stablizer} and \ref{alg:fixed-points} are easy to implement and quick to run. The GitHub repository \url{https://github.com/Michael-Howes/BurnsideProcess} contains an implementation of the Sylow--Burnside process in the Julia Programming Language \cite{bezanson2017julia}. This implementation is used to perform the simulations in Section~\ref{sec:sim}.

\subsubsection{Sampling from a stabilizer}
Equation \eqref{eq:equiv} implies that the map \(g \mapsto (\sigma g \sigma^{-1},g)\) is a bijection between \(H \cap \sigma^{-1} H \sigma\) and \((H \times H)_\sigma\). Thus, to perform the first step it is sufficient to be able to sample from \(H \cap \sigma^{-1}H\sigma\). By definition
\[
    H \cap \sigma^{-1}H \sigma = \langle\eta_1,\ldots,\eta_k \rangle \cap \langle \sigma^{-1}\eta_1\sigma,\ldots, \sigma^{-1}\eta_k \sigma\rangle.
\]
Furthermore, the permutations \(\sigma^{-1}\eta_j\sigma\) are all cyclic permutations of order \(p\). Thus, if \(\sigma^{-1} \eta_j \sigma \in H\) for some \(j\), then it must be the case that \(\sigma^{-1}\eta_j\sigma=\eta_{j'}^i\) for some \(j' \in [1:k]\) and \(i \in [1:p-1]\). It follows that if \(A = \{j : \sigma\eta_j \sigma^{-1}\in  H\}\), then
\[
    H \cap \sigma^{-1}H\sigma = \langle \eta_j : j \in A\rangle.
\]
And thus every element of \(H \cap \sigma^{-1}H\sigma\) can be written uniquely as \(\prod_{j \in A}\eta_j^{i_j}\) for some choice of \(i_j \in [0:(p-1)]\). From this observation, it is easy to draw uniform samples from \(H \cap \sigma^{-1}H\sigma\) and hence \((H \times H)_{\sigma}\). This idea is formalized in Algorithm~\ref{alg:stablizer}.

\subsubsection{Sampling from a set of fixed points}

Now we will describe how to sample \(\tau\) uniformly from \(S_{pk}^{h,g}\). First note that during the Burnside process, we always have that \(\sigma \in S_{pk}^{h,g}\), and so we can assume that \(S_{pk}^{h,g} \neq \emptyset\). This implies that \(h = \sigma g \sigma^{-1}\) for some \(\sigma\) and thus \(h\) has the same cycle type of \(g\). Since \(h,g \in H\), \(h\) and \(g\) only have fixed points and \(p\)-cycles, and they must have the same number of each.

Let \(y \in [0:k]\) be the number of \(p\)-cycles in \(h\) and \(g\). A permutation, \(\tau\) satisfies \(\tau g \tau^{-1}=h\) if and only if \(\tau\) maps the \((k-y)p\) fixed points of \(g\) to the \((k-y)p\) fixed points of \(h\) and maps the \(y\) \(p\)-cycles of \(g\) to the \(y\) \(p\)-cycles of \(h\). Specifically, \(\tau\) can act as an arbitrary bijection on the fixed points but on the \(p\)-cycles \(\tau\) can only cyclically shift the elements of the \(p\)-cycles of \(g\) and then map the cycles as blocks to the cycles of \(h\). It follows that 
\begin{equation}\label{eq:size-commutator}
    |S_{pk}^{h,g}|=((k-y)p)! \times y! p^{y}.
\end{equation}
Algorithm~\ref{alg:fixed-points} is derived from formula and description above.

\begin{algorithm*}[p]
    \caption{Uniform sampler for the stabilizer \((H \times H)_\sigma\).}
    \begin{algorithmic}[1]
        \Require{\(\sigma \in S_{pk}\).}
        \Ensure{\((h,g) \in (H \times H)_\sigma\) uniformly sampled.}
        \State{\(g \gets \id\)}\Comment{The identity permutation.}
        \For{\(j \in [1:k]\)}
        \If{\(\sigma \eta_j\sigma^{-1} \in H\)}
        \State{Sample \(i_j\) uniformly from \([0:p-1]\)}
        \State{\(g \gets g\eta_j^{i_j}\)}
        \EndIf
        \EndFor
        \State{\(h \gets \sigma g \sigma^{-1}\)}
        \State\Return{\((h,g)\)}
    \end{algorithmic}
    \label{alg:stablizer}
\end{algorithm*}

\begin{algorithm*}[p]
    \caption{Uniform sampler for the set of fixed points \(S_{pk}^{h,g}\).}
    \begin{algorithmic}[1]
        \Require{\((h,g) \in H \times H\) with \(S_{pk}^{h,g}\) non-empty.}
        \Ensure{\(\tau \in S_{pk}^{h,g}\) uniformly sampled.}
        \Statex\Comment{Define \(\tau\) on the fixed points of \(g\).}
        \State{\(F(g) \gets \{a : g(a)=a\}\)}
        \State{\(F(h) \gets \{b : h(b)=b\}\)}
        \State{Uniformly sample a bijection \(\rho : F(g) \to F(h)\)}
        \For{\(a \in F(g)\)}
        \State{\(\tau(a) \gets \rho(a)\)}
        \EndFor
        \Statex\Comment{Define \(\tau\) on the \(p\)-cycles of \(g\).}
        \State{Find the \(p\)-cycles of \(g\): \((a_{1,1},a_{1,2},\ldots,a_{1,p}),\ldots,(a_{y,1},\ldots,a_{y,p})\)}
        \State{Find the \(p\)-cycles of \(h\): \((b_{1,1},b_{1,2},\ldots,b_{1,p}),\ldots,(b_{y,1},\ldots,b_{y,p})\)}
        \State{Uniformly sample \(c_1,\ldots,c_y \in [0:(p-1)].\)}
        \State{Uniformly sample \(\gamma \in S_y\).}
        \For{\(i \gets 1\) to \(y\)}
        \For{\(j \gets 1\) to \(p\)}
        \State{\(\tau(a_{i,j}) \gets b_{\gamma(i), j+c_i \bmod p}\)}
        \EndFor
        \EndFor
        \State\Return{\(\tau\).}
    \end{algorithmic}
    \label{alg:fixed-points}
\end{algorithm*}

\subsection{Proof of Theorem~\ref{thrm:lump}}\label{sec:lump-proof}

Algorithms~\ref{alg:stablizer} and \ref{alg:fixed-points} will be used to show that the Sylow--Burnside process lumps to double coset size. Let \(P\) be the transition kernel for the Sylow--Burnside process and let \(T : S_{pk} \to [k:2k]\) be given by \(|H\sigma H|=p^{T(\sigma)}\). Theorem~\ref{thrm:lump} states that \(P\) lumps under \(T\) and gives expressions for the lumped kernel \(\overline{P}\) and stationary distribution \(\overline{\pi}\). By Dynkin's criteria \cite[Section~6.3]{kemeny1976finite} it is sufficient to show that for all \(a,b \in [k:2k]\) if \(T(\sigma)=a\), then
\begin{align}\label{eq:dynkin}
    \overline{P}_a(b) & := P_\sigma(\{\tau: T(\tau)=b\})=\sum_{y=0}^{2k-\max\{a,b\}}\binom{2k-a}{y} \left(\frac{p-1}{p}\right)^y  \frac{p^{a+b-2k}}{(p(k-y))!}f(b-y;k-y), \\
    \overline{\pi}(b) & := \pi(\{\tau : T(\tau)=b\}) =\frac{f(b;k)}{Z},\label{eq:dynkin-2}
\end{align}
where \(f(a;k)\) is the number of double cosets in \(H \backslash S_{pk}\slash H\) of size \(p^{a}\) and \(Z\) is the total number of double cosets. To show \eqref{eq:dynkin}, we will introduce a second function \(R:H \to [0:k]\) and then prove two lemmas about the relationship between \(R\) and \(T\). To define \(R(g)\) write \(g=\eta_1^{i_1}\cdots\eta_k^{i_k}\) as in \eqref{eq:H-form} and set
\begin{equation}
    \label{eq:R}
    R(g) = |\{j \in [1:k] : i_j \neq 0\}|.
\end{equation}
That is, \(R(g)\) is the number of \(p\)-cycles in \(g\). The first lemma, Lemma~\ref{lem:dist-R}, describes the distribution of \(R(g)\) when \(g\) is uniformly sampled from \(H \cap \sigma^{-1} H \sigma \) as in Algorithm~\ref{alg:stablizer}.
\begin{lemma}\label{lem:dist-R}
    Fix \(\sigma \in S_{pk}\) such that \(T(\sigma)=a\). If \(g\) is uniformly sampled from \(H \cap \sigma^{-1} H \sigma\), then 
    \[
        \Prob(R(g) = y) = \binom{2k-a}{y} \left(\frac{p-1}{p}\right)^y \left(\frac{1}{p}\right)^{2k-a-y}.
    \]
    In words, \(R(g)\) is binomially distributed with parameters \(2k-a\) and \(\frac{p-1}{p}\).
\end{lemma}
\begin{proof}
    By Algorithm~\ref{alg:stablizer}, we know that \(g\) can be represented as
    \[ g = \prod_{j \in A} \eta_j^{i_j},\]
    where \(A = \{ j : \sigma \eta_j \sigma^{-1} \in H\}\) and \(\{i_j : j \in A\}\) are independent and uniformly distributed on \([0:(p-1)]\). Furthermore, we have that
    \[
        p^{|A|} =  |H \cap \sigma^{-1}H\sigma| = \frac{|H|^2}{|H\sigma H|} = p^{2k-a}.
    \]
    Thus, \(R(g)=\sum_{j \in A}I[i_j\neq 0]\) is equal in distribution to the sum of \(2k-a\) independent binary variables each with expected value \(\frac{p-1}{p}\). Thus, \(R(g)\) is binomially distributed with parameters \(2k-a\) and \(\frac{p-1}{p}\) as claimed.
\end{proof}

The next result, Lemma~\ref{lem:dist-T}, gives the distribution of \(T(\tau)\) when \(\tau\) is sampled from \(S_{pk}^{h,g}\) as in Algorithm~\ref{alg:fixed-points}. It shows that the distribution of \(T(\tau)\) can be described in terms of uniformly sampling \(\tau' \in S_{pk'}\) with \(k' \le k\) and looking at the size of the \(p\)-Sylow double coset containing \(\tau'\).
\begin{lemma}
    \label{lem:dist-T}
    Let \((h,g) \in H\times H\) be such that \(S_{pk}^{h,g} \neq \emptyset\) and \(R(g)=R(h)=y\). Let \(\tau\) be uniformly distributed in \(S_{pk}^{(h,g)}\), then
    \[
        \Prob(T(\tau) = b) = 
        \begin{cases}
            \frac{1}{(p(k-y))!}p^{b-y}f(b-y;k-y) & \text{if } b \le 2k-y, \\
            0                                    & \text{otherwise}.
        \end{cases}
    \]
\end{lemma}
\begin{proof}
    First note that
    \[
        T(\tau) = 2k - |\{j : \tau \eta_j \tau^{-1} \in H\}|.
    \]
    Next write \(g=\eta_1^{j_1}\cdots \eta_k^{i_k}\) and set \(J_0(g):=\{j \in [1:k]: \eta_j =0\}\) and \(J_1(g) := [1:k] \setminus J_0(g)\).  Define \(J_0(h)\) and \(J_1(h)\) analogously.
    
    By assumption, we have that \(|J_1(g)|=|J_1(h)|=y\) and \(|J_0(g)|=|J_0(h)|=k-y\). In particular, \(g\) and \(h\) both have \((k-y)p\) fixed points and \(y\) \(p\)-cycles. By Algorithm~\ref{alg:fixed-points} we know that \(\tau\) must map each cycle \(\eta_j^{i_j}\) with \(i_j \neq 0\) to a power of a \(p\)-cycle of \(h\). It therefore follows that if \(j \in J_1(g)\), then \(\tau\eta_j\tau^{-1} \in H\). Therefore,
    \[
        T(\tau) \le 2k-|J_1(g)| = 2k-y.
    \]
    It remains to show that if \(b \le 2k-y\), then
    \[
        \Prob(T(\tau) = b)= \frac{1}{(p(k-y))!}p^{b-y}f(b-y;k-y).
    \]
    To see this, let \(F(g)\) and \(F(h)\) be the set of fixed points of \(g\) and \(h\). Recall that by Algorithm~\ref{alg:fixed-points}, \(\tau\) induces a random bijection from \(F(g)\) to \(F(h)\). It follows that the number of \(j \in J_1(h)\) such that \(\tau\eta_j\tau^{-1}\in H\) is equal in distribution \(|\{j \in [1:(k-y)] : \rho \eta_j \rho^{-1} \in H'\}|\) where \(\rho\) is uniformly sampled form \(S_{p(k-y)}\) and \(H'\) is a \(p\)-Sylow subgroup of \(S_{p(k-y)}\). It follows that for \(b \le 2k-y\),
    \begin{align*}
        \Prob(T(\tau)=b) & =\Prob(2k - |\{j \in J_0(h) : \tau \eta_j \tau^{-1} \in H\}|+|\{j \in J_1(h) : \tau \eta_j \tau^{-1} \in H\}| = b) \\
                         & =\Prob(2k -y- |\{j \in [1:(k-y)] : \rho \eta_j \rho^{-1} \in H\}| = b)                                             \\
                         & =\Prob(2(k -y)- |\{j \in [1:(k-y)] : \rho \eta_j \rho^{-1} \in H\}| = b-y)                                         \\
                         & =\Prob(|H'\rho H'|=p^{b-y})                                                                                        \\
                         & =\frac{1}{((k-y)p)!}p^{b-y}f(b-y;k-y),
    \end{align*}
    where the last line holds because there are \(p^{b-y}f(b-y;k-y)\) permutations in \(S_{(k-y)p}\) in double cosets of size \(p^{b-y}\).
\end{proof}
We are now ready to prove Theorem~\ref{thrm:lump} by showing \eqref{eq:dynkin} and \eqref{eq:dynkin-2}. This is done by marginalizing over the distribution of \(R(g)\) where \(h,g\) are sampled in the first step of the Sylow--Burnside process. Lemma~\ref{lem:dist-R} gives the distribution of \(R(g)\) and Lemma~\ref{lem:dist-T} gives the conditional distribution of \(T(\tau)\) given \(R(g)\). The distribution of \(T(\tau)\) when \(\tau \sim P_\sigma\) is derived by combining these two results.
\begin{proof}[Proof of Theorem~\ref{thrm:lump}]
    To prove \eqref{eq:dynkin}, fix \(\sigma \in S_{pk}\) with \(T(\sigma)=a\). To sample \(\tau \sim P_\sigma\), we first sample \((h,g)\) as in Algorithm~\ref{alg:stablizer} and then sample \(\tau\) as in Algorithm~\ref{alg:fixed-points}. Thus,
    \[
        \Prob(T(\sigma)=b)  =\sum_{y=0}^{k}\Prob(R(g)=y)\Prob(T(\tau)=b\mid R(g)=y) 
    \]
    By Lemmas~\ref{lem:dist-R} and \ref{lem:dist-T},
    \begin{align*}
         & \sum_{y=0}^{k}\Prob(R(g)=y)\Prob(T(\tau)=b\mid R(g)=y)                                                                             \\
         & =\sum_{y=0}^{k}\binom{2k-a}{y} \left(\frac{p-1}{p}\right)^y \left(\frac{1}{p}\right)^{2k-a-y} \frac{1}{(p(k-y))!}p^{b-y}f(b-y;k-y) \\
         & =\sum_{y=0}^{2k-\max\{a,b\}}\binom{2k-a}{y} \left(\frac{p-1}{p}\right)^y  \frac{p^{a+b-2k}}{(p(k-y))!}f(b-y;k-y).
    \end{align*}
    Thus, \(P\) lumps under the map \(T\) and the lumped kernel is given by \eqref{eq:dynkin}. For the stationary distribution of \(\overline{P}\), suppose that \(\tau \sim \pi\), then
    \[  
        \overline{\pi}(b) = \Prob(T(\tau)=b) = \frac{1}{Zp^{b}}p^{b}f(b;k) = \frac{f(b;k)}{Z},
    \]
    and so \(\overline{\pi}\) is given by \eqref{eq:dynkin-2}.
\end{proof}

\section{Proofs of mixing time results}\label{sec:proofs}

In this section we first prove the non-asymptotic result in Theorem~\ref{thrm:bound} and then take a limit to derive the limit profile in Theorem~\ref{thrm:mix}.

\subsection{Proof of Theorem~\ref{thrm:bound}}

As explained in Section~\ref{sec:proof-overview}, there are two ideas behind the proof of Theorem~\ref{thrm:bound}. The first idea is to approximate the mixing time of \(P\) with the mixing time of \(\overline{P}\). This approximation is based on the fact that most double cosets are large and is proved in Lemmas~\ref{lem:conditional} and \ref{lem:tv-comparison}. 

The second idea is that the mixing time of \(\overline{P}\) is approximately the time it takes reach \(2k\) (corresponding to the Burnside process reaching the large double cosets). To compute the time it takes \(\overline{P}\) to reach \(2k\) we approximate \(\overline{P}\) by a simpler transition kernel \(Q\). In Lemma~\ref{lem:Qapprox} we show that \(Q\) is a good approximation to \(\overline{P}\). Finally, in Lemma~\ref{lem:Qmix}, we compute the time it takes \(Q\) to reach \(2k\). Combining this calculation with the various approximations gives Theorem~\ref{thrm:bound}.

The first lemma describes the conditional distribution of \(\tau\) when \(\tau\) is sampled via the Burnside process and conditioned on being in a large double coset. Specifically, it states that if \(\tau \sim P^t_{\sigma}\) and \(T(\tau)=2k\), then \(\tau\) is uniformly distributed over the set of such \(\tau\).
\begin{lemma}\label{lem:conditional}
    Let \(B=\{\sigma \in S_{pk}:T(\sigma)=2k\}\) and fix \(\sigma \in S_{pk}\) and \(t \ge 1\). Then, the conditional distributions \(P^t_{\sigma}(\cdot \mid B)\) and \(\pi(\cdot \mid B)\) are both uniform on \(B\). That is, for any, \(A \subseteq S_{pk}\)
    \[
        P^t_{\sigma}(A \mid B)=\pi(A\mid B)=\frac{|A \cap B|}{|B|}.
    \]
\end{lemma}
\begin{proof}
    Note that \(\pi(\tau)= Z^{-1}p^{-T(\tau)}\). Thus, \(\pi(\tau)\) depends only on \(T(\tau)\) and so \(\pi(\cdot \mid B)\) is uniform on \(B\).
    
    To show that \(P^t_{\sigma}(A \mid B) = |A\cap B|/|B|\) we can assume without loss of generality that \(t=1\). This is because for any \(t \ge 1\)
    \[
        P^t_{\sigma}(A \mid B) = \sum_{\tau \in S_{pk}} P^{t-1}_{\sigma}(\tau)P_{\tau}(A \mid B).
    \]
    Thus, if \(P_{\sigma}(\cdot \mid B)\) is the uniform distribution for every \(\sigma \in S_{pk}\), then \(P^t_{\sigma}(\cdot\mid B)\) is also the uniform distribution. When \(t=1\), we have that
    \[
        P_{\sigma}(\tau) = \frac{1}{|(H \times H)|_{\sigma}}\sum_{(h,g) \in (H \times H)_{\sigma} \cap (H \times H)_\tau}\frac{1}{|S_{pk}^{h,g}|}.
    \]
    If \(T(\tau)=2k\), then \((H \times H)_\tau\) only contains the identity. Thus, for \(\tau \in B\),
    \[
        P_{\sigma}(\tau) = \frac{1}{|(H \times H)_{\sigma}|} \frac{1}{|S_{pk}|}.
    \]
    This shows that \(\tau \mapsto P_{\sigma}(\tau)\) is constant on \(B\) and hence the conditional distribution \(P_\sigma(\cdot \mid B)\) is uniform.
\end{proof}

Lemma~\ref{lem:tv-comparison} is used to compare the mixing time of the lumped chain to the mixing time of the original process. The key idea is that the distance \(\Vert P^t_\sigma - \pi \Vert_{\mathrm{TV}}\) is well approximated by \(\pi(B)-P^t_\sigma(B)\). This approximation is justified by the fact that \(\pi(B) \approx 1\) and by Lemma~\ref{lem:conditional} which states that \(P^t_\sigma(\cdot \mid B)=\pi(\cdot \mid B)\). The proof below formalizes this idea.
\begin{lemma}\label{lem:tv-comparison}
    Fix \(p \ge 11\),  \(t \ge 1\) and \(\sigma \in S_{pk}\). If \(T(\sigma)=a\), then
    \[
        \Vert \overline{P}_{a}^t - \overline{\pi} \Vert_{\mathrm{TV}} \le \Vert P_\sigma^t - \pi \Vert_{\mathrm{TV}} \le \Vert \overline{P}_a^t - \overline{\pi} \Vert_{\mathrm{TV}} + \frac{2p^4}{(p-1)!}
    \]
\end{lemma}
\begin{proof}
    The first inequality holds for any lumping of a Markov chain. To prove the second inequality, let \(B=\{\tau : T(\tau)=2k\}\) as in Lemma~\ref{lem:conditional}. Also let \(\varepsilon = \frac{2p^4}{(p-1)!}\). By Proposition~\ref{prop:most-large} we have 
    \[
        \pi(B^c) = \frac{Z-f(2k;k)}{Z} \le \varepsilon.
    \]
    Thus,
    \begin{align*}
        \Vert P_\sigma^t - \pi \Vert_{\mathrm{TV}} & =\sup_{A \subseteq S_{pk}} \{\pi(A)-P_\sigma^t(A)\}                                                       \\
                                                   & =\sup_{A \subseteq S_{pk}} \{\pi(A \cap B)+\pi(A \cap B^c)-P_\sigma^t(A\cap B) - P_\sigma^t(A \cap B^c)\} \\
                                                   & \le\sup_{A \subseteq S_{pk}} \{\pi(A \cap B)-P_\sigma^t(A\cap B)\} + \varepsilon.
    \end{align*}
    By Lemma~\ref{lem:conditional}, for any \(A \subseteq S_{pk}\),
    \begin{align*}
        \pi(A \cap B)-P_\sigma^t(A\cap B) & = \pi(B)\pi(A \mid B) - P_\sigma^t(B)P_\sigma(A \mid B)          \\
                                          & = (\pi(B)-P^t_\sigma(B))\frac{|A \cap B|}{|B|}                   \\
                                          & \le \pi(B)-P^t_\sigma(B)                                         \\
                                          & = \overline{\pi}(2k) - \overline{P}^t_a(2k)                      \\
                                          & \le \Vert \overline{\pi} - \overline{P}^t_a\Vert_{\mathrm{TV}}.
    \end{align*}
    Thus,
    \[ 
        \Vert P_\sigma^t - \pi \Vert_{\mathrm{TV}} \le \sup_{A \subseteq S_{pk}} \{\pi(A \cap B)-P_\sigma^t(A\cap B)\}+\varepsilon \le \Vert \overline{\pi}-\overline{P}_a^t\Vert_{\mathrm{TV}}+\varepsilon.
    \]
\end{proof}

Let \(Q\) be the following transition kernel on \([k:2k]\)
\begin{equation}\label{eq:Q}
    Q(a,b) = \begin{cases}\binom{2k-a}{2k-b}\left(\frac{p-1}{p}\right)^{2k-b}\left(\frac{1}{p}\right)^{b-a} & \text{if } b \ge a, \\
             0                                                                                 & \text{otherwise}.
    \end{cases}
\end{equation}
The transition kernel \(Q\) is upper triangular and \(2k\) is the unique absorbing state. This implies that the stationary distribution of \(Q\) is \(\delta_{2k}\)--a point pass at \(2k\). The next lemma states that transition matrix \(Q\) is a good approximation to the lumped transition matrix \(\overline{P}\). This approximation is based on the fact that most double cosets are large and so the sum in \eqref{eq:lump-P} is dominated by the largest term.

\begin{lemma}\label{lem:Qapprox}
    For all \(t \in \naturals\) and \(a \in [k:2k]\)
    \[
        \Vert \overline{P}^t_a-Q^t_a \Vert_{\mathrm{TV}} \le \frac{t}{(p-2)!}.
    \]
\end{lemma}
\begin{proof}
    By induction, it suffices to show that for all \(a \in [k:2k]\),
    \[
        \Vert \overline{P}_a- Q_a \Vert_{\mathrm{TV}} \le \frac{1}{(p-2)!}.
    \]
    By Theorem~\ref{thrm:lump} we have that for \(b \in [k:2k]\) with \(b \ge a\)
    \begin{align*}
        \overline{P}(a,b) & =\sum_{y=0}^{2k-b}\binom{2k-a}{y} \left(\frac{p-1}{p}\right)^y \frac{p^{a+b-2k} }{(p(k-y))!}f(b-y;k-y) \\
                          & \ge \binom{2k-a}{2k-b} \left(\frac{p-1}{p}\right)^{2k-b}  \frac{p^{a+b-2k}}{(p(b-k))!}f(2(b-k);b-k)
    \end{align*}
    By Theorem~\ref{thrm:most-large}, we have that \(\frac{1}{(p(b-k))!}f(2(b-k);b-k) \ge \frac{1}{p^{2(b-k)}}\left(1-\frac{1}{(p-2)!}\right)\). Thus,
    
    \begin{align*}       
        \overline{P}(a,b) & \ge \binom{2k-a}{2k-b}\left(\frac{p-1}{p}\right)^{2k-b}p^{a+b-2k}\frac{1}{p^{2(b-k)}}\left(1-\frac{1}{(p-2)!}\right) \\
                          & =\binom{2k-a}{2k-b}\left(\frac{p-1}{p}\right)^{2k-b}\left(\frac{1}{p}\right)^{b-a}\left(1-\frac{1}{(p-2)!}\right)    \\
                          & =Q(a,b)\left(1-\frac{1}{(p-2)!}\right).
    \end{align*}
    And for \(b < a\), \(\overline{P}(a,b) \ge 0 = Q(a,b)\). Thus,
    \begin{align*}
        \Vert \overline{P}_a -Q_a \Vert_{\mathrm{TV}}                                   = & \sum_{b=0}^k(Q(a,b)-\overline{P}(a,b))_+                                  \\
        =                                                                                 & \sum_{b=a}^k(Q(a,b)-\overline{P}(a,b))_+                                  \\
        \le                                                                               & \sum_{b=a}^k\left(Q(a,b) - \left(1-\frac{1}{(p-2)!}\right)Q(a,b)\right)_+ \\
        =                                                                                 & \frac{1}{(p-2)!}\sum_{b=a}^k Q(a,b)                                       \\
        =                                                                                 & \frac{1}{(p-2)!}
    \end{align*}
\end{proof}
The next lemma states that \(\overline{\pi}\) is well-approximated by \(\delta_{2k}\). That is the stationary distributions of \(\overline{P}\) and \(Q\) are close to each other.
\begin{lemma}\label{lem:pi_approx}
    If \(p \ge 11\), then 
    \[
        \Vert \overline{\pi}-\delta_{2k}\Vert_{\mathrm{TV}} \le \frac{2 p^4}{(p-1)!}
    \]
\end{lemma}
\begin{proof}
    By Proposition~\ref{prop:most-large} we have
    \[
        \Vert \overline{\pi}-\delta_{2k}\Vert_{\mathrm{TV}} = 1-\overline{\pi}(2k) = \frac{Z-f(2k;k)}{Z} \le \frac{2p^4}{(p-1)!}.
    \]
\end{proof}
The chain \(Q\) is much easier to analyze than \(\overline{\pi}\). In particular, the time it takes \(Q\) to reach \(2k\) is exactly equal in distribution to the maximum of several independent geometric random variables.
\begin{lemma}
    \label{lem:Qmix}
    For all \(t \in \naturals\) and \(a \in [k:2k]\)
    \[
        \Vert Q_a^t - \delta_{2k}\Vert_{\mathrm{TV}} = 1 - Q_a^t(2k) = 1-\left(1-\left(1-\frac{1}{p}\right)^t\right)^{2k-a}.
    \]
\end{lemma}
\begin{proof}
    We will prove this result by giving an equivalent description of the chain \(Q\). Suppose that \((X_t)_{t \ge 0}\) is a Markov chain with transition kernel \(Q\) and \(X_0=a\). Then, given \(X_t=b\), \(X_{t+1}\) is equal in distribution to \(a\) plus a binomial variable with parameters \(2k-b\) and \(1/p\). The process \(X_t\)
    \begin{enumerate}
        \item Set \(X_0=a\) and start with \(2k-a\) independent coins each with probability of heads equal to \(1/p\).
        \item For each \(t\ge t\), flip all the coins available coins and remove the ones that land on heads. Let \(X_t\) be equal to \(X_{t-1}\) plus the number of removed coins.
    \end{enumerate}
    From the above description, \(X_t=2k\) is equivalent to all \(2k-a\) coins landing on heads in the first \(t\) steps. This gives the equality
    \begin{equation}\label{eq:geoms}
        \Prob(X_{t}=2k) =\Prob(\max\{Z_j : 1 \le j \le 2k-a\} \le t),
    \end{equation}
    where \(Z_j\) is the time it takes the \(j\)th coin to land on heads. Thus, \(\{Z_j\}_{j=1}^{2k-a}\) are independent geometric random variables with \(\Prob(Z_j > t) = \left(1-\frac{1}{p}\right)^t\). 
    
    By independence, we have
    \begin{align*}
        \Vert Q^t_0 - \delta_{2k} \Vert_{\mathrm{TV}} & = 1-Q_0^t(2k)                                           \\
                                                      & = 1 - \Prob(X_t=2k)                                     \\
                                                      & = 1 - \Prob(\max\{Z_j : 1\le j \le 2k-a\} \le t)        \\
                                                      & =1-\Prob(Z_1\le t)^{2k-a}                               \\
                                                      & =1-\left(1-\left(1-\frac{1}{p}\right)^t\right)^{2k-a},
    \end{align*}
    as claimed.
\end{proof}
We are now ready to prove Theorem~\ref{thrm:bound}. That is, for all \(\sigma \in S_{pk}\) and \(t \ge 1\), if \(T(\sigma)=a\) and \(p \ge 11\), then
\begin{equation}\label{eq:bound-variable-sigma}
    \left\vert \Vert P^t_\sigma -\pi\Vert_{\mathrm{TV}} - 1 + \left(1-\left(1-\frac{1}{p}\right)^t\right)^{2k-a}\right\vert \le \frac{4p^4}{(p-1)!} + \frac{t}{(p-2)!}.
\end{equation}
\begin{proof}[Proof of Theorem~\ref{thrm:bound}]
    Note that
    \[
        \left\vert \Vert P^t_\sigma -\pi\Vert_{\mathrm{TV}} - \Vert Q_a^t - \delta_{2k}\Vert_{\mathrm{TV}}\right\vert \le \left\vert \Vert P^t_\sigma -\pi\Vert_{\mathrm{TV}} - \Vert \overline{P}_a^t - \overline{\pi}\Vert_{\mathrm{TV}}\right\vert+\left\vert \Vert \overline{P}_a^t - \overline{\pi}\Vert_{\mathrm{TV}} - \Vert Q_a^t - \delta_{2k}\Vert_{\mathrm{TV}}\right\vert,
    \]
    and, by Lemma~\ref{lem:tv-comparison}
    \[
        \left\vert \Vert P^t_\sigma -\pi\Vert_{\mathrm{TV}} - \Vert \overline{P}_a^t - \overline{\pi}\Vert_{\mathrm{TV}}\right\vert \le \frac{2p^4}{(p-1)!}.
    \]
    Furthermore, by Lemma~\ref{lem:Qapprox} and \ref{lem:pi_approx}, we have
    \begin{align*}
        \left\vert \Vert \overline{P}_a^t - \overline{\pi}\Vert_{\mathrm{TV}} - \Vert Q_a^t - \delta_{2k}\Vert_{\mathrm{TV}}\right\vert 
        \le & \Vert \overline{P}_a^t - Q_a^t \Vert_{\mathrm{TV}} + \Vert \overline{\pi}-\delta_{2k}\Vert_{\mathrm{TV}} \\
        \le & \frac{t}{(p-2)!} + \frac{2p^4}{(p-1)!}.
    \end{align*}
    Thus,
    \[
        \left\vert \Vert P^t_\sigma -\pi\Vert_{\mathrm{TV}} - 1 + \left(1-\left(1-\frac{1}{p}\right)^t\right)^{2k-a}\right\vert = \left\vert \Vert P^t_\sigma -\pi\Vert_{\mathrm{TV}} - \Vert Q_a^t - \delta_{2k}\Vert_{\mathrm{TV}}\right\vert \le \frac{4p^4}{(p-1)!} + \frac{t}{(p-2)!},
    \]
    As claimed in Theorem~\ref{thrm:bound}.
\end{proof}

\subsection{Proof of Theorem~\ref{thrm:mix}}

Theorem~\ref{thrm:mix} is a straight forward consequence of Theorem~\ref{thrm:bound} and the following lemma.

\begin{lemma}\label{lem:Qlimit}
    Let \(p=p_m\) and \(k=k_m\) be two sequences with \(p_m \to \infty\). Then,
    \begin{enumerate}
        \item If \(k_m \to k< \infty\) and \(c \ge 0\), then
              \[
                  \left(1-\left(1-\frac{1}{p}\right)^{\lfloor pc\rfloor}\right)^k = (1-e^{-c})^k.
              \]
        \item If \(k_m \to \infty\) and \(c \in \reals\), then
              \[
                  \left(1-\left(1-\frac{1}{p}\right)^{\lfloor p\log k + pc\rfloor}\right)^{k} = \exp(-e^{-c}).
              \]
    \end{enumerate}
\end{lemma}
\begin{proof}
    The first claim is immediate. For the second claim note that for any \(p \ge 2\) and \(x \ge 0\), 
    \begin{align*}
        1-\exp\left(-\frac{x-1}{p}\right) \le 1-\left(1-\frac{1}{p}\right)^{\lfloor x\rfloor} \le 1-\exp\left(-x\left(\frac{1}{p}-\frac{1}{p^2}\right)\right)
    \end{align*}
    If \(x = p\log k + pc\) and \(p,k\to\infty\) with \(k \le p-1\), then
    \begin{align*}
        \left(1-\exp\left(-\frac{x-1}{p}\right)\right)^k             & =\left(1-\frac{e^{-c-1/p}}{k}\right)^k \to \exp(-e^{-c})              \\
        \left(1-\exp\left(-\frac{x}{p}-\frac{x}{p^2}\right)\right)^k & =\left(1-\frac{e^{-c-(\log k + c)/p}}{k}\right)^k \to \exp(-e^{-c}),
    \end{align*}
    And so \(\left(1-\left(1-\frac{1}{p}\right)^{\lfloor x\rfloor}\right)^{k} \to \exp({-e^{-c}})\).
\end{proof}

\begin{proof}[Proof of Theorem~\ref{thrm:mix}]
    Theorem~\ref{thrm:bound} gives
    \begin{align*}
        \left| \Vert P^t_\sigma -\pi\Vert_{\mathrm{TV}}- 1 + \left(1-\left(1-\frac{1}{p}\right)^t\right)^{2k-a}\right| \le \frac{4p^4}{(p-1)!} + \frac{t}{(p-2)!},
    \end{align*}
    where \(a = T(\sigma)\). Since \(a \in [k:2k]\), this implies that
    \begin{align*}
        \left| d(t)- 1 + \left(1-\left(1-\frac{1}{p}\right)^t\right)^{k}\right| \le \frac{4p^4}{(p-1)!} + \frac{t}{(p-2)!}.
    \end{align*}
    Next note that if \(t = \lfloor cp\rfloor \) or \(t = \lfloor p\log k + cp\rfloor\) for some fixed \(c\), then \(\frac{4p^4}{(p-2)!}+\frac{t}{(p-2)!} \to 0\) as \(p \to \infty\). Thus, for \(t =\lfloor cp\rfloor\) or \(t=\lfloor p \log k + cp\rfloor\), we have that
    \[
        \lim_{p \to \infty} d(t) =1-\lim_{p \to \infty}\left(1-\left(1-\frac{1}{p}\right)^t\right)^{k}.
    \]
    Thus, Lemma~\ref{lem:Qlimit} implies Theorem~\ref{thrm:mix}.
\end{proof}

\section{Examples and simulations}\label{sec:examples}

In this section, Theorems~\ref{thrm:mix} and \ref{thrm:bound} are compared to an exact computation in the case when \(k=1\) and to simulations with moderate values of \(p\) and \(k\).

\subsection{Exact computation when \(k=1\)}

In the case when \(k=1\), all Sylow \(p\)-double cosets either have size \(p\) or size \(p^2\). Theorem~\ref{thrm:formula} from \cite{diaconis2025number} shows that there are \(n_1=p-1\) double cosets of size \(p\) and \(n_2 = \frac{(p-1)!-(p-1)}{p}\) of size \(p^2\). The stationary distribution of the Sylow--Burnside process in this case is
\[
    \pi(\sigma) = \begin{cases}
        \frac{1}{(n_1+n_2)p^2} & \text{if }|H\sigma H|=p^2, \\
        \frac{1}{(n_1+n_2)p}   & \text{if } |H\sigma H|=p.
    \end{cases}
\]
In Section~\ref{appn:exact} of the appendix, it is shown that the entries of the transition kernel are given by
\[
    P(\sigma,\tau) = \begin{cases}
        \frac{1}{p!}                  & \text{if } |H\sigma H| = p^2,                                   \\
        \frac{1}{pp!}                 & \text{if } |H\sigma H| = p \text{ and } H\sigma H \neq H\tau H, \\
        \frac{p-1}{p^2}+\frac{1}{pp!} & \text{if } |H\sigma H|=p \text{ and } H\sigma H = H\tau H.
    \end{cases}
\]
Section~\ref{appn:exact} also contains an eigenvalue decomposition of \(P\). The transition kernel has four distinct eigenvalues. These eigenvalues are \(1\), \(1-\frac{1}{p}\), \(1-\frac{1}{p}-\frac{p-1}{p(p-2)!}\) and \(0\) with multiplicities \(
1, p-2, 1\) and \(p!-p\) respectively. From the eigenvalue decomposition, it is possible to exactly compute \(\Vert P^t_\sigma - \pi \Vert_{\mathrm{TV}}\). Section~\ref{appn:exact} proves that for all \(\sigma \in S_p\)
\begin{equation}
    \Vert P^t_\sigma - \pi \Vert_{\mathrm{TV}} =\begin{cases}
        \left(1-\frac{1}{n_1}\right)\left(1-\frac{1}{p}\right)^t + \frac{n_2}{n_1(n_1+n_2)}\left(1-\frac{1}{p}-\frac{p-1}{p(p-2)!}\right)^t & \text{if } |H\sigma H| = p.  \\
        \frac{n_1}{n_1+n_2}\left(1-\frac{1}{p}-\frac{p-1}{p(p-2)!}\right)^t                                                                 & \text{if } |H\sigma H|=p^2.
    \end{cases}\label{eq:exact}
\end{equation}
If \(p\) is even moderately large, then \(n_2 \gg n_1\) and \((p-2)! \gg p\). This gives
\[
    \frac{n_2}{n_1(n_1+n_2)}\left(1-\frac{1}{p}-\frac{p-1}{p(p-2)!}\right)^t \approx \frac{1}{n_1}\left(1-\frac{1}{p}\right)^t, \quad \frac{n_1}{n_1+n_2}\approx 0.
\]
Thus, in the \(k=1\) case we can derive the approximation:
\begin{align*}
    \Vert P^t_\sigma - \pi \Vert_{\mathrm{TV}} & \approx \begin{cases}
                                                             \left(1-\frac{1}{p}\right)^t & \text{if } |H\sigma H|=p,    \\
                                                             0                            & \text{if } |H\sigma H|=p^2.
                                                         \end{cases}
\end{align*}
Which agrees with the approximation in Theorem~\ref{thrm:bound} when \(k=1\).
\subsection{Simulations}\label{sec:sim}
By running the algorithms in Section~\ref{sec:running} we can empirically study the Sylow--Burnside process and estimate \(\Vert P_\sigma^t - \pi \Vert_{\mathrm{TV}}\) for moderate \(p\) and \(k\). Figure~\ref{fig:sim} shows the result of one such simulation. In the figure, \(\Vert P_\sigma^t - \pi \Vert_{\mathrm{TV}}\) is estimated by generating \(B\) independent realizations of the Sylow--Burnside process initialized at the identity permutation and run for \(t_{\max}\) steps. This gives \(B\) sequences of permutations \((\sigma_t^{(b)})_{t=1}^{t_{\max}}\) for \(b=1,\ldots,B\). For each \(t\), the permutations \(\sigma_t^{(b)}\) are used to construct an empirical measure \(\hat\mu_t\) on \([k:2k]\) given by
\[
    \hat\mu_t(a) = \frac{1}{B}\sum_{b=1}^B I[T(\sigma_t^{(b)})=a].
\]
The measure \(\hat\mu_t\) is an approximation to the measure \(\overline{P}^t_k\)--the probability distribution induced by taking \(t\) steps of the lumped Burnside process started at \(k\). Figure~\ref{fig:sim} shows \(\Vert \hat\mu_t - \overline{\pi}\Vert_{\mathrm{TV}}\) as a function of \(t\) and approximations based on Theorems~\ref{thrm:mix} and \ref{thrm:bound}. The approximation from Theorem~\ref{thrm:bound} is especially accurate. 

\begin{figure}
    \centering
    \begin{subfigure}{0.48\textwidth}
        \centering
        \includegraphics[width=\textwidth]{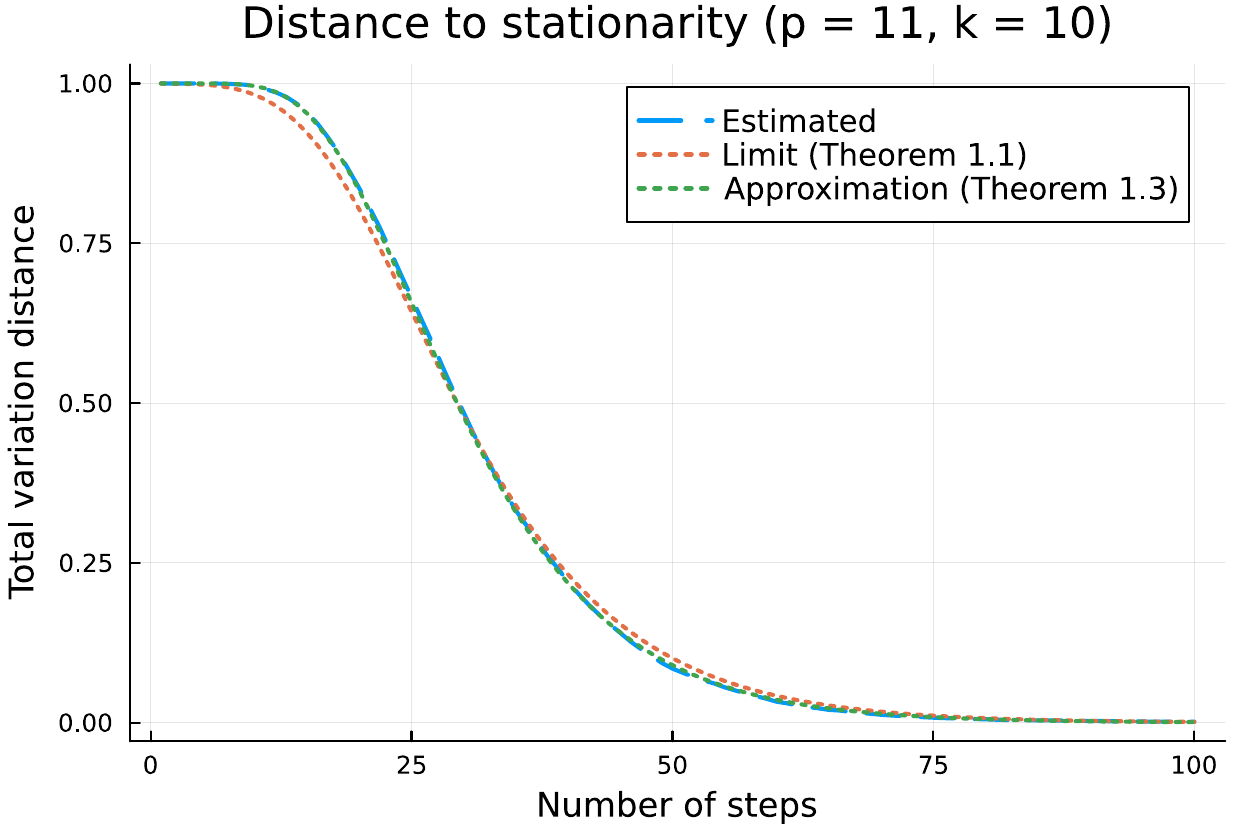}
        \caption{}
    \end{subfigure}
    \hfill
    \begin{subfigure}{0.48\textwidth}
        \centering
        \includegraphics[width=\textwidth]{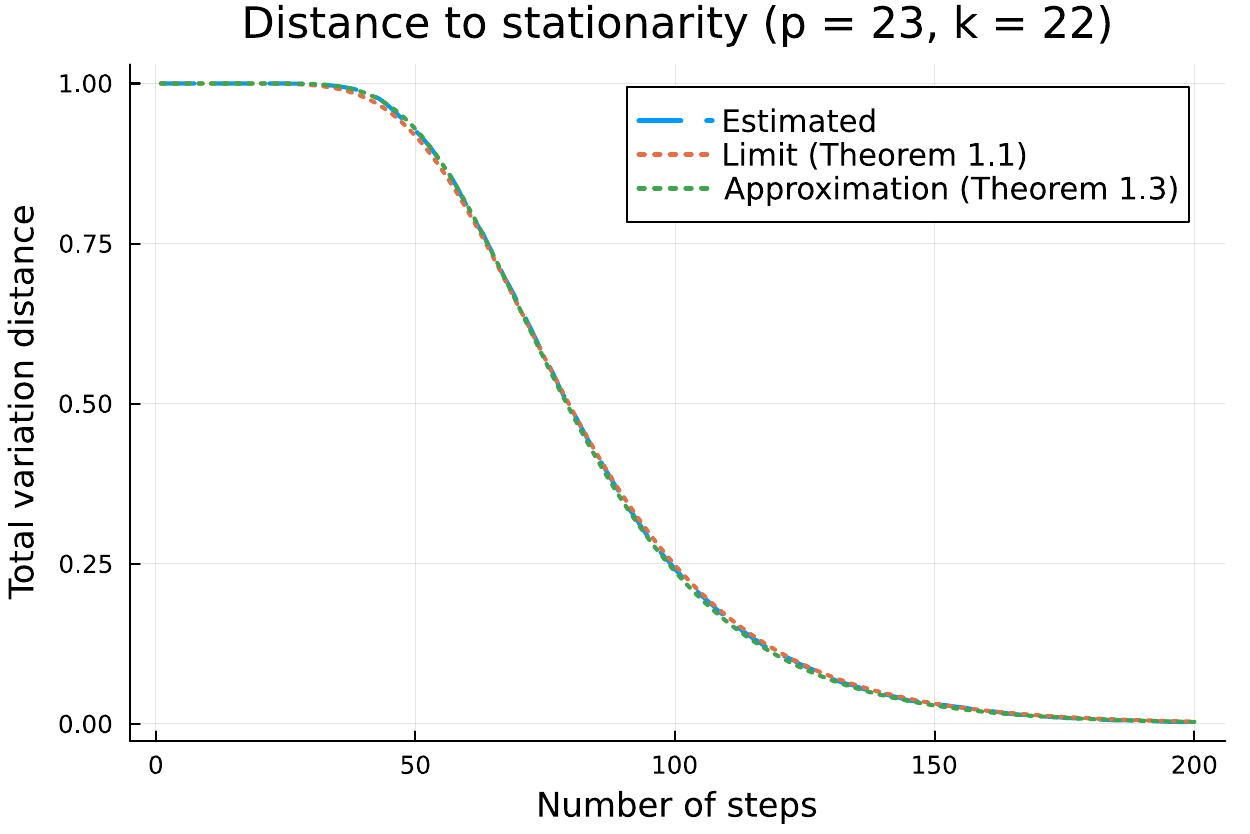}
        \caption{}
    \end{subfigure}
    \caption{Both figures show a plot of \(\Vert \hat\mu_t - \overline{\pi}\Vert_{\mathrm{TV}}\) as a function of \(t\) for different values of \(p\) and \(k\). In both cases, the empirical measure \(\hat\mu_t\) is estimated based on \(B=10,000\) runs of the Sylow--Burnside process. The plots also show the limit profile \(t \mapsto 1-\exp(-e^{-c})\) where \(c = (t-p\log k)/p\) from Theorem~\ref{thrm:mix} and the approximation \(t \mapsto 1-(1-(1-1/p)^t)^k\) from Theorem~\ref{thrm:bound}. Both approximations are accurate when \(p=23\) and \(k=22\) but only Theorem~\ref{thrm:bound} is accurate for \(p=11\) and \(k=10\). The high accuracy of Theorem~\ref{thrm:bound} is reflected in the fact that the error term goes to zero super-exponential fast in \(p\). The code use to create Figure~\ref{fig:sim} is available at \url{https://github.com/Michael-Howes/BurnsideProcess/blob/main/Examples/sylow_double_cosets_example.ipynb}}
    \label{fig:sim}
\end{figure}

\section{Related Markov chains}\label{sec:related}

The present argument gives a new approach to the analysis of Burnside processes. Previous efforts used Doeblin conditions \cite{jerrum1993uniform, diaconis2005analysis}, explicit diagonalization \cite{diaconis2020hahn, diaconis2025curiously} and coupling \cite{aldous2014reversible, paguyo2022mixing, rahmani2022mixing}. 

The approach used here is most similar to that of \emph{strong stationary times} \cite{aldous1986shuffling, diaconis1990strong}. Lemmas~\ref{lem:conditional} and Lemma~\ref{lem:tv-comparison} can be restated as follows. Let \((\sigma_t)_{t \ge 0}\) be generated according to the Sylow--Burnside process and let \(T\) be the smallest \(t\) such that  \(\sigma_t\) is in a large double coset. Then, the random permutation \(\sigma_T\) is independent of \(T\) and approximately distributed according to \(\pi\). If \(\sigma_T\) was exactly distributed according to the stationary distribution \(\pi\), then \(T\) would be a strong stationary time and we would have the bound
\begin{equation} \label{eq:stationary-time}
    \Vert P_{\sigma}^t - \pi \Vert_{\mathrm{TV}} \le \Prob(T > t \mid \sigma_0=\sigma)
\end{equation}
see \cite{aldous1986shuffling, diaconis1990strong} or \cite[Chapter~6]{Levin}. Since \(\sigma_T\) is only approximately distributed according to \(\pi\), \eqref{eq:stationary-time} holds up to a small error. Controlling this error and approximating \(\Prob(T > t\mid \sigma_0=\sigma)\) is essentially how Theorem~\ref{thrm:bound} was proved.

The analysis of the Sylow--Burnside process is also reminiscent of a natural Markov chain for sampling \(k\)-dimesional subspaces of an \(n\)-dimension vector space over the finite field \(\mathbb{F}_q\). This Markov chain was analyzed in \cite{d1995nearest}. The Markov chain lumps distance to the initial subspace and most of the subspaces are at the maximal distance \(k\). The distance typically increases and so \(k\) steps are necessary and sufficient. It would be interesting to investigate the connection with \cite{d1995nearest} further. One direction would be to study the Burnside process for Sylow \(p\)-double cosets in the matrix group \(GL_n(\mathbb{F}_q)\).

\section{The non-Abelian case}\label{sec:non-abelian}

We end with some conjectures and questions about the behavior of the Sylow--Burnside process in the non-Abelian case when \(n \ge p^2\). Even answering these questions for \(n=p^2\) is an interesting problem.

The main challenge appears to be generalizing Section~\ref{sec:running} on the implementation of the Burnside process. Although some aspects of Section~\ref{sec:running} go through, it seems difficult to describe the subgroup \(H \cap \sigma^{-1} H\sigma\). If there is a method for uniformly sampling from \(H \cap \sigma^{-1}H\sigma\), then the algorithms in Section~\ref{sec:running} can be easily adapted to run the Sylow--Burnside process in general. Unfortunately, even computing \(|H \cap \sigma^{-1} H \sigma|\) seems hard.

One important property about the Abelian case was a simple description of the distribution of the cycle type of \(g\) when \(g\) was uniformly sampled from \(H \cap \sigma H \sigma^{-1}\) (Lemma~\ref{lem:dist-R}). This leads to the following question. 

\begin{question}\label{ques:intersects}
    Let \(H\) be a Sylow \(p\)-subgroup of \(S_n\). Does there exist a simple description of the cycle type of \(g\) when \(g\) is uniformly sampled from \(H \cap \sigma^{-1}H\sigma\)? Relatedly, is there a formula for \(|H\cap \sigma^{-1}H\sigma|\) or a simple algorithm for sampling from \(H \cap \sigma^{-1}H\sigma\)?
\end{question}

Assuming that one can sample from \(H \cap \sigma^{-1}H\sigma\) and hence run the Sylow--Burnside process, we have the following conjectures about the behavior of the Markov chain. These two conjectures correspond to the two main tools used in the Abelian case as described in Section~\ref{sec:proof-overview}.

\begin{conjecture}\label{conj:lump}
    Let \(P\) be the transition kernel for the Sylow--Burnside process and let \(T\) be a function that records the conjugacy type of \(H \cap \sigma^{-1}H\sigma\) as a subgroup of \(S_n\), then \(P\) lumps under \(T\). 
\end{conjecture}
\begin{conjecture}\label{conj:most-large}
    Let \(Z\) be the total number of Sylow \(p\)-double cosets in \(S_n\) and let \(Z_{0}\) be the number of double cosets of maximal size \(p^{2m}\), then if \(p,n\to \infty\)
    \[
        \frac{Z_{0}}{Z} \to 1
    \]
\end{conjecture}
\begin{remark}
    \begin{enumerate}
        \item In the Abelian case, the conjugacy type of \(H \cap \sigma^{-1}H\sigma\) is exactly determined by the size of \(|H\sigma H|\). When \(n \ge p^2\), we believe that Conjecture~\ref{conj:lump} is the correct generalization of Theorem~\ref{thrm:lump}.
        \item Conjecture~\ref{conj:most-large} is similar to Theorem~1.2 in \cite{diaconis2025number} which can be equivalently stated as \(\frac{Z_0p^{2m}}{n!}\to 1\) as \(n \to \infty\). The difference between Conjecture~\ref{conj:most-large} and Theorem~1.2 in \cite{diaconis2025number} is the same as the difference between Proposition~\ref{prop:most-large} and Theorem~\ref{thrm:most-large}.
    \end{enumerate}
\end{remark}
In addition to proving Conjectures~\ref{conj:lump} and \ref{conj:most-large}, a generalization of the current paper would require a way of approximating the time it takes the Sylow--Burnside process to reach the large double cosets. In the Abelian case, this approximation was done with the monotone transition kernel \(Q\). Assuming that Conjecture~\ref{conj:lump} is true, there may exist an analogous transition matrix that satisfies a similar monotonicity property. We state a more precise version of this an open question.
\begin{question}\label{ques:monotone}
    Call a transition kernel \(Q\) on \(S_n\) \emph{monotone} if \(Q(\sigma,\tau) > 0\) implies \(|H\sigma H| \le |H\tau H|\).    Does there exist a monotone approximation to the transition kernel for the Sylow--Burnside process?
\end{question}
Answers to these questions and proofs of the conjectures seem to be the most promising path to generalizing beyond the Abelian case.

\newcommand{\etalchar}[1]{$^{#1}$}

\begin{appendix}
    \section{Proof of Proposition~\ref{prop:most-large}}\label{sec:additional-proofs}
    To prove Proposition~\ref{prop:most-large} we will prove a series of lemmas that bound \(f(a;k)\). As stated in Section~\ref{sec:sylow-background}, these proofs are based on the comments given in \cite[Remark~3.4]{diaconis2025number}. 
    
    \begin{lemma}\label{lemma:first-bound}
        For all \(a,p,k\)
        \[
            f(a;k) \le \frac{1}{p^{a}}((a-k)p)!(2k-a)!\binom{k}{2k-a}^2(p(p-1))^{2k-a}.
        \]
    \end{lemma}
    \begin{proof}
        By equation \eqref{eq:numcosets} we have
        \begin{equation}\label{eq:alt-sum}
            f(a;k) = \frac{1}{p^a}\sum_{j=2k-a}^{k} (-1)^{j-(2k-a)}\Gamma_{j,a},
        \end{equation}
        where
        \[
            \Gamma_{j,a} = ((k-j)p)!j!\binom{k}{j}^2 (p(p-1))^j \binom{j}{2k-a}.
        \]
        For \(j \in [2k-a:k-1]\) we have
        \begin{align*}
            \frac{\Gamma_{j+1,a}}{\Gamma_{j,a}} & =\frac{((k-j-1)p)!(j+1)!\binom{k}{j+1}^2 (p(p-1))^{j+1} \binom{j+1}{2k-a}}{((k-j)p)!j!\binom{k}{j}^2 (p(p-1))^j \binom{j}{2k-a}} \\
                                                & =\frac{(j+1)\left(\frac{k-j}{j+1}\right)^2p(p-1)\frac{j+1}{j+1-2k-a}}{((k-j)p)((k-j)p-1)\cdots ((k-j-1)p+1)}                     \\
                                                & =\frac{(k-j)^2p(p-1)}{(j+1-2k-a)((k-j)p)((k-j)p-1)\cdots ((k-j-1)p+1)}                                                           \\
                                                & =\frac{(k-j)(p-1)}{(j+1-2k-a)((k-j)p-1)\cdots ((k-j-1)p+1)}.
        \end{align*}
        Since \(j < k\), we have that \((k-j)(p-1) \le (k-j)p -1\). It follows that
        \begin{align*}
            \frac{\Gamma_{j+1,a}}{\Gamma_{j,a}} & \le \frac{1}{(j+1-2k-a)}{((k-j)p-2)\cdots((k-j-1)p+1)} \le \frac{1}{(j+1-2k-a)(p-2)!} \le 1.
        \end{align*}
        This means that \(0 \le \Gamma_{j+1,a} \le \Gamma_{j,a}\) and so the sum in \eqref{eq:alt-sum} is an alternating sum of non-negative, non-increasing terms. This implies that
        \begin{equation}\label{eq:gamma-bound}
            f(a;k) \le \frac{1}{p^a}\Gamma_{2k-a,a} = \frac{1}{p^a}((a-k)p)!(2k-a)!\binom{k}{2k-a}^2 (p(p-1))^{2k-a}.
        \end{equation}
    \end{proof}
    The next lemma gives a bound on \(f(a;k)\) that is independent of \(a\).
    \begin{lemma}
        \label{lem:second-bound}
        For \(p \ge 11\) and \(k \le a < 2k\),
        \begin{equation}\label{eq:f-small}
            f(a;k) \le \frac{1}{p^{2k-2}}((k-1)p)!k^2(p-1).
        \end{equation}
    \end{lemma}
    \begin{proof}
        Let \(\Lambda_a =\frac{1}{p^a}((a-k)p)!(2k-a)!\binom{k}{2k-a}^2 (p(p-1))^{2k-a}\) and note that by Lemma~\ref{lemma:first-bound} we have that \(f(a;k) \le \Lambda_a\). Furthermore,
        \[
            \Lambda_{2k-1} = \frac{1}{p^{2k-1}}((k-1)p)!k^2p(p-1) = \frac{1}{p^{2k-2}}((k-1)p)!k^2(p-1) 
        \]
        is the right-hand-side of \eqref{eq:f-small}. Therefore, for the current lemma, it suffices to show that \((\Lambda_a)_{a=k}^{2k-1}\) is a decreasing sequence as this would imply
        \[
            f(a;k) \le \Lambda_a \le \Lambda_{2k-1}.
        \]
        Note that
        \begin{align*}
            \frac{\Lambda_{a+1}}{\Lambda_a} & =\frac{p^a((a+1-k)p)!(2k-a-1)!\binom{k}{2k-a-1}^2 (p(p-1))^{2k-a-1}}{p^{a+1}((a-k)p)!(2k-a)!\binom{k}{2k-a}^2 (p(p-1))^{2k-a}} \\
                                            & =\frac{((a+1-k)p)!\left(\frac{2k-a}{a+1-k}\right)^2}{((a-k)p)!(2k-a) p^2(p-1)}                                                 \\
                                            & =\frac{((a+1-k)p)((a+1-k)p-1)\cdots ((a-k)p+1)(2k-a)}{(a+1-k) p^2(p-1)}                                                        \\
                                            & =\frac{((a+1-k)p-1)\cdots ((a-k)p+1)(2k-a)}{(a+1-k)p(p-1)}                                                                     \\
                                            & \ge \frac{(p-1)!(2k-a)}{(a+1-k)p(p-1)}                                                                                         \\
                                            & =\frac{(p-2)!(2k-a)}{(a+1-k)p}
        \end{align*}
        We know that \(a+1-k \le k < p\) and \(2k-a \ge 1\). Thus, \(\frac{\Lambda_{a+1}}{\Lambda_a} \ge \frac{(p-2)!}{p^2}\). If \(p \ge 11\), then \((p-2)! \ge p^2\) and so \(\Lambda_{a+1}\ge \Lambda_a\) as required.
    \end{proof}
    
    The final lemma compares \(f(a;k)\) to \(f(2k;k)\).
    \begin{lemma}
        \label{lem:third-bound}
        For \(p \ge 11\) and \(k \le a < 2k\),
        \begin{equation}
            \frac{f(a;k)}{f(2k;k)} \le \frac{2p^3}{(p-1)!}
        \end{equation}
    \end{lemma}
    \begin{proof}
        By Lemma~\ref{lem:second-bound} we know that
        \[
            f(a;k) \le \frac{1}{p^{2k-2}}((k-1)p)!k^2(p-1)
        \]
        Furthermore, from Theorem~\ref{thrm:most-large} we have
        \[
            f(2k;k) \ge \left(1-\frac{1}{(p-2)!}\right)\frac{(pk)!}{p^{2k}} \ge \frac{(pk)!}{2p^{2k}}.
        \]
        It follows that
        \begin{align*}
            \frac{f(a;k)}{f(2k;k)} & \le \frac{2((k-1)p)!k^2(p-1)p^{2k}}{p^{2k-2}(pk)!} \\
                                   & \le \frac{2((k-1)p)!pk^2(p-1)}{(pk)!}              \\
                                   & = \frac{2pk^2(p-1)}{(pk)(pk-1)\cdots (p(k-1)-1)}   \\
                                   & \le \frac{2p^4}{(pk)(pk-1)\cdots (p(k-1)-1)}       \\
                                   & \le \frac{2 p^4}{p!}                               \\
                                   & =\frac{2p^3}{(p-1)!}.
        \end{align*}
    \end{proof}
    Finally, we can prove Proposition~\ref{prop:most-large}.
    \begin{proof}[Proof of Proposition~\ref{prop:most-large}]
        By Lemma~\ref{lem:third-bound} if \(Z=\sum_{a=k}^{2k}f(a;k)\), then
        \begin{align*}
            f(2k;k) & =Z-\sum_{a=k}^{2k-1}f(a;k)            \\
                    & \ge Z-(k-1)\frac{2p^3}{(p-1)!}f(2k;k) \\
                    & \ge Z-\frac{2p^4}{(p-1)!}Z, 
        \end{align*}
        And so
        \[
            1-\frac{2p^4}{(p-1)!} \le \frac{f(2k;k)}{Z} \le 1,
        \]
        as required.
    \end{proof}
    \section{Derivation of total variation distance when \(k=1\)}\label{appn:exact}
    In the section we prove \eqref{eq:exact} by proving a series of lemmas. We first prove the formula for \(P\) given in Section~\ref{sec:examples}, then we compute an eigenvalue decomposition for \(P\). This eigenvalue decomposition is used to give an exact formula for \(P^t(\sigma,\tau)\) which is then used to compute the total variation distance as in \eqref{eq:exact}.
    \begin{lemma}
        \label{lem:P_k_eq_1}
        Suppose that \(k=1\), then the transition kernel for the Sylow--Burnside process is given by 
        \[
            P(\sigma,\tau) = \begin{cases}
                \frac{1}{p!}                  & \text{if } |H\sigma H| = p^2,                                   \\
                \frac{1}{pp!}                 & \text{if } |H\sigma H| = p \text{ and } H\sigma H \neq H\tau H, \\
                \frac{p-1}{p^2}+\frac{1}{pp!} & \text{if } |H\sigma H|=p \text{ and } H\sigma H = H\tau H.
            \end{cases}
        \]
    \end{lemma}
    \begin{proof}
        The matrix \(P\) is given by
        \[
            P(\sigma,\tau) = \frac{1}{(H \times H)_\sigma}\sum_{h,g \in (H \times H)_\sigma \cap (H \times H)_\tau} \frac{1}{|S_p^{h,g}|}
        \]
        If \(|H\sigma H|=p^2\), then \((H \times H)_\sigma\) only contains the identity and thus
        \[
            P(\sigma,\tau) = \frac{1}{1}\times \frac{1}{|S_p^{\id,\id}|}=\frac{1}{p!}.
        \]
        If \(|H\sigma H|=p\) and \(H\sigma H \neq H\tau H\), then \((H \times H)_\sigma \cap (H \times H)_\tau\) again only contains the identity and so
        \[
            P(\sigma,\tau) = \frac{1}{|(H \times H)_\sigma}\times \frac{1}{|S_p^{\id,\id}|}=\frac{1}{pp!}.
        \]
        Finally, if \(H \times H = H\tau H\) and \(|H\sigma H|=p\), then \(|(H \times H)_\sigma \cap (H \times H)_\tau|=p\) and
        \begin{align*}
            P(\sigma,\tau) & = \frac{1}{|(H \times H)_\sigma}\left(\frac{1}{|S_p^{\id,\id}|}+\sum_{(h,g) \in (H \times H)_\sigma \setminus \{(\id,\id)\}}\frac{1}{|S_p^{h,g}|}\right) \\
                           & =\frac{1}{p}\left(\frac{1}{p!}+(p-1)\frac{1}{p}\right)                                                                                                   \\
                           & =\frac{p-1}{p^2}+\frac{1}{pp!}.
        \end{align*}
    \end{proof}
    The following notation is useful for describing the eigenvectors of \(P\). Let \(A \subseteq S_p\) be the set of permutations in double cosets of size \(p\) and let \(B = S_p \setminus A\). We will fix \(\rho_1,\ldots,\rho_{n_1}\) such that 
    \[
        A =\bigsqcup_{j=1}^{n_1} H\rho_j H.
    \]
    Finally, for a subset \(C \subseteq S_p\), let \(\ones_C\) be the vector that is \(1\) on \(C\) and \(0\) otherwise.
    \begin{lemma}\label{lem:eigendecomp}
        Let \(P\) be the transition kernel for the Sylow--Burnside process when \(k=1\). Then \(P\) his diagonalizable and the eigenvalues, eigenvectors and their multiplicities are given by Table~\ref{table:eigens}.
    \end{lemma}
    \begin{table}[h]
        \begin{center}
            \begin{tabular}{c|c|p{5.5cm}|p{5cm}}
                Eigenvalues                             & Multiplicity   & Right Eigenspace                                                                             & Left Eigenspace                                                \\
                \hline 
                1                                       & 1              & Spanned by the constant function \(\ones_{S_p}\).                                            & Spanned by the stationary distribution \(\pi\)                 \\
                \(1-\frac{1}{p}\)                       & \(n_1-1\)      & Functions of the form \(\sum c_j \ones_{H \rho_j H}\) with \(\sum c_j = 0\).                 & Same as the right eigenspace.                                  \\
                \(1-\frac{1}{p} - \frac{p-1}{p(p-2)!}\) & 1              & Spanned by \(\frac{1}{n_1}\ones_A -\frac{1}{n_2}\ones_B\)                                    & Spanned by \(\frac{1}{pn_1}\ones_A - \frac{1}{p^2n_2}\ones_B\) \\
                0                                       & \(p! - n_1-1\) & Functions that are orthogonal \(\ones_B\) and \(\ones_{H \rho_j H}\) for \(j=1,\ldots,n_1\). & Same as the right eigenspace.
            \end{tabular}
        \end{center}
        \caption{The eigenvalues and eigenspaces of \(P\).}
        \label{table:eigens}
    \end{table}
    \begin{proof}
        All transition kernels have \(1\) as an eigenvalue, the constant function as a right eigenvector and the stationary distribution as a left eigenvector. Now suppose that \(\psi = \sum_{j=1}^{n_1}c_j\ones_{H\rho_j H}\) with \(\sum_{j=1}^{n_1}c_j=0\). From Lemma~\ref{lem:P_k_eq_1} we can see that both of the maps \(\sigma \mapsto P(\sigma,\tau)\) and \(\tau \mapsto P(\sigma,\tau)\) are constant on double cosets. Thus,
        \begin{align*}
            (P\psi)(\sigma) & =\sum_{j=1}^{n_1}\sum_{\tau \in H \rho_j H}P(\sigma,\tau)c_j              \\
                            & =\sum_{j=1}^{n_1}pP(\sigma,\rho_j)c_j                                     \\
                            & =\sum_{j=1}^{n_1}\frac{1}{p!}c_j + I[H\sigma H=H\rho_j H]\frac{p-1}{p}c_j \\
                            & =\left(1-\frac{1}{p}\right)\psi(\sigma).
        \end{align*}
        A similar argument also shows that \((\psi P)(\tau)=\left(1-\frac{1}{p}\right)\psi(\tau)\). Next let \(\psi = \frac{1}{n_1}\ones_A -\frac{1}{n_2}\ones_B\). Then if \(|H\sigma H|=p\)
        \begin{align*}
            (P\psi)(\sigma) & =\sum_{\tau \in A}\frac{1}{n_1}P(\sigma,\tau)-\sum_{\tau \in B}\frac{1}{n_2}P(\sigma,\tau)     \\
                            & =\frac{1}{n_1}\left(\frac{pn_1}{pp!}+\frac{p(p-1)}{p^2}\right)-\frac{1}{n_2}\frac{n_2p^2}{pp!} \\
                            & = \left(\frac{n_1}{p!}+\frac{p-1}{p}-\frac{n_1p}{p!}\right)\frac{1}{n_1}                       \\
                            & =\left(\frac{p-1}{p!}+\frac{p-1}{p}-\frac{(p-1)p}{p!}\right)\psi(\sigma)                       \\
                            & =\left(1-\frac{1}{p}-\frac{p-1}{p(p-2)!}\right)\psi(\sigma).
        \end{align*}
        Where the second last line holds because \(n_1=p-1\) and \(\psi(\sigma)=\frac{1}{n_1}\). If \(\sigma \in B\), then \(\psi(\sigma)=-\frac{1}{n_2}\). Recall also that \(n_2=\frac{(p-1)!-(p-1)}{p}\), then
        \begin{align*}
            (P\psi)(\sigma) & =\sum_{\tau \in A}\frac{1}{n_1}P(\sigma,\tau)-\sum_{\tau \in B}\frac{1}{n_2}P(\sigma,\tau) \\
                            & =\frac{pn_1}{n_1}\frac{1}{p!}-\frac{n_2p^2}{n_2}\frac{1}{p!}                               \\
                            & =\left(\frac{n_2}{(p-1)!}-\frac{n_2p}{(p-1)!}\right)\frac{1}{n_2}                          \\
                            & =\left(-\frac{(p-1)!-(p-1)}{p!}+\frac{(p-1)!-(p-1)}{(p-1)!}\right)\psi(\sigma)             \\
                            & =\left(\frac{-(p-1)!+(p-1)+p!-p(p-1)}{p!}\right)\psi(\sigma)                               \\
                            & =\left(1-\frac{1}{p}-\frac{p-1}{p(p-2)!}\right)\psi(\sigma).
        \end{align*}
        And so \(\psi\) is a right eigenvector with eigenvalue \(1-\frac{1}{p}-\frac{p-1}{p(p-2)!}\). A similar argument shows that \(\tilde{\psi}=\frac{1}{pn_1}\ones_A-\frac{1}{p^2n_2}\ones_B\) is a left eigenvector with the same eigenvalue. Finally, the functions \(\sigma \mapsto P(\sigma,\tau)\) and \(\tau \mapsto P(\sigma,\tau)\) are constant on \(B\) and each of the double cosets \(H\rho_jH\). It follows that if \(\psi\) is orthogonal to \(\ones_B\) and \(\ones_{H \rho_j H}\) for all \(j\), then \(P\psi = \psi P = 0\). There are \(p!-n_1-1\) such vectors, and so we have found an eigenvalue decomposition of \(P\).
    \end{proof}
    \begin{lemma}\label{lem:powers-of-P}
        Let \(P\) be the transition kernel for the Sylow--Burnside process when \(k=1\) and let \(\tilde\psi = \frac{1}{pn_1}\ones_A-\frac{1}{p^2n_2}\ones_B\). If \(|H\sigma H|=p\), then
        \begin{align}\label{eq:P-power-small}
            P^t_\sigma & =   \left(1-\frac{1}{p}\right)^t\left(\frac{1}{p}\ones_{H \sigma H}-\frac{1}{pn_1}\ones_A\right)            +\frac{n_2}{n_1+n_2}\left(1-\frac{1}{p}-\frac{p-1}{p(p-2)!}\right)^t\tilde\psi+\pi,
        \end{align}
        and if \(|H\sigma H|=p^2\), then
        \begin{equation}\label{eq:P-power-large}
            P^t_\sigma = -\frac{n_1}{n_1+n_2}\left(1-\frac{1}{p}-\frac{p-1}{p(p-2)!}\right)^t\tilde\psi+\pi.
        \end{equation}
    \end{lemma}
    \begin{proof}
        Let \(\ones_\sigma = \ones_{\{\sigma\}}\) and note that \(P^t_\sigma =\ones_\sigma P^t\). Our goal is thus to write \(\ones_\sigma\) in terms of the left eigenvectors of \(P\) and then apply Lemma~\ref{lem:eigendecomp}. Suppose that \(|H\sigma H|=p\), then
        \begin{align*}
            \ones_\sigma & =\left(\ones_\sigma - \frac{1}{p}\ones_{H \sigma H}\right) + \frac{1}{p}\left(\ones_{H \sigma H}-\frac{1}{n_1}\ones_A\right)+\frac{1}{pn_1}\ones_A.
        \end{align*}
        And
        \begin{align*}
            \frac{1}{pn_1}\ones_A & =\left(\frac{n_2}{n_1+n_2}\frac{1}{pn_1}\ones_A+\frac{1}{p(n_1+n_2)}\ones_A\right)+\left(-\frac{n_2}{n_1+n_2}\frac{1}{p^2n_2}\ones_B +\frac{1}{p^2(n_1+n_2)}\ones_B\right) \\
                                  & =\frac{n_2}{n_1+n_2}\left(\frac{1}{pn_1}\ones_A-\frac{1}{p^2n_2}\ones_B \right)+\left(\frac{1}{pn_1}\ones_A+\frac{1}{p^2n_2}\ones_B\right).
        \end{align*}
        Recall that \(\tilde{\psi}=\frac{1}{pn_1}\ones_A-\frac{1}{p^2n_2}\ones_B\) and \(\frac{1}{pn_1}\ones_A+\frac{1}{p^2n_2}\ones_B=\pi\). We thus have
        \[
            \ones_\sigma = \left(\ones_\sigma - \frac{1}{p}\ones_{H \sigma H}\right) + \frac{1}{p}\left(\ones_{H \sigma H}-\frac{1}{n_1}\ones_A\right)+\frac{n_2}{n_1+n_2}\tilde{\psi}+\pi.
        \]
        The vector \(\ones_\sigma - \frac{1}{p}\ones_{H \sigma H}\) is a left eigenvector of \(P\) with eigenvalue 0 and \(\ones_{H \sigma H}-\frac{1}{n_1}\ones_A\) is a left eigenvector of \(P\) with eigenvalue \(1-\frac{1}{p}\). Finally, \(\tilde{\psi}\) and \(\pi\) are left eigenvectors with eigenvalues \(1-\frac{1}{p}-\frac{p-1}{p(p-2)!}\) and 1 respectively. It follows that
        \begin{align*}
            P^t_\sigma & =\ones_\sigma P^t                                                                                                                                                                   \\
                       & =\left(\left(\ones_\sigma - \frac{1}{p}\ones_{H \sigma H}\right) + \frac{1}{p}\left(\ones_{H \sigma H}-\frac{1}{n_1}\ones_A\right)+\frac{n_2}{n_1+n_2}\tilde{\psi}+\pi\right)P^t    \\
                       & =\left(1-\frac{1}{p}\right)^t\left(\frac{1}{p}\ones_{H \sigma H}-\frac{1}{pn_1}\ones_A\right)+\frac{n_2}{n_1+n_2}\left(1-\frac{1}{p}-\frac{p-1}{p(p-2)!}\right)^t\tilde{\psi}+\pi.
        \end{align*}
        Since \(\tilde{\psi}=\frac{1}{pn_1}\ones_A-\frac{1}{p^2n_2}\ones_B\). We have shown \eqref{eq:P-power-small}. Now suppose that \(|H\sigma H|=p^2\), then a similar calculation gives
        \[
            \ones_\sigma = \left(\ones_\sigma - \frac{1}{n_2p}\ones_B\right)-\frac{n_1}{(n_1+n_2)}\tilde{\psi}+\pi.
        \]
        The vector \(\ones_\sigma - \frac{1}{n_2p}\) is a left eigenvector of \(P\) with eigenvalue \(0\) and hence
        \begin{align*}
            P^t_\sigma & =\ones_\sigma P^t                                                                                        \\
                       & =\left(\left(\ones_\sigma - \frac{1}{n_2p}\ones_B\right)-\frac{n_1}{(n_1+n_2)}\tilde{\psi}+\pi\right)P^t \\
                       & =-\frac{n_1}{n_1+n_2}\left(1-\frac{1}{p}-\frac{p-1}{p(p-2)!}\right)^t\tilde{\psi}+\pi.
        \end{align*}
        Since \(-\tilde{\psi}=\frac{1}{p^2n_2}\ones_B-\frac{1}{pn_1}\ones_A\), we have also shown \eqref{eq:P-power-large}.
    \end{proof}
    \begin{lemma}
        Let \(P\) be the transition kernel for the Sylow--Burnside process when \(k=1\). Let \(\pi\) be the stationary distribution for \(P\). Then for all \(t \ge 1\)
        \[
            \Vert P^t_\sigma - \pi \Vert_{\mathrm{TV}} =\begin{cases}
                \left(1-\frac{1}{n_1}\right)\left(1-\frac{1}{p}\right)^t + \frac{n_2}{n_1(n_1+n_2)}\left(1-\frac{1}{p}-\frac{p-1}{p(p-2)!}\right)^t & \text{if } |H\sigma H| = p. \\
                \frac{n_1}{n_1+n_2}\left(1-\frac{1}{p}-\frac{p-1}{p(p-2)!}\right)^t                                                                 & \text{if }|H\sigma H|=p^2.\end{cases}
        \]
    \end{lemma}
    \begin{proof}
        Suppose \(|H\sigma H|=p\). Then, Lemma~\ref{lem:powers-of-P} shows that \(P^t_\sigma(\tau) > \pi(\tau)\) if and only if \(\tau \in H\sigma H\) we have
        \begin{align*}
            P^t_\sigma(\tau)-\pi(\tau) & =\left(1-\frac{1}{p}\right)^t\left(1-\frac{1}{n_1}\right)\frac{1}{p}+\frac{n_2}{pn_1(n_1+n_2)}\left(1-\frac{1}{p}-\frac{p-1}{p(p-2)!}\right)^t.
        \end{align*}
        Since \(|H\sigma H|=p\) it follows that
        \begin{align*}
            \Vert P^t_\sigma-\pi\Vert_{\mathrm{TV}} & =\sum_{\tau \in H\sigma H}\left(1-\frac{1}{p}\right)^t\left(1-\frac{1}{n_1}\right)\frac{1}{p}+\frac{n_2}{pn_1(n_1+n_2)}\left(1-\frac{1}{p}-\frac{p-1}{p(p-2)!}\right)^t \\
                                                    & =\left(1-\frac{1}{p}\right)^t\left(1-\frac{1}{n_1}\right)+\frac{n_2}{n_1(n_1+n_2)}\left(1-\frac{1}{p}-\frac{p-1}{p(p-2)!}\right)^t
        \end{align*}
        If \(|H\sigma H|=p^2\), then Lemma~\ref{lem:powers-of-P} gives \(P^t_\sigma(\tau) > \pi(\tau)\) if and only if \(\tau \in B\). Furthermore, for \(\tau \in B\)
        \begin{align*}
            P^t_\sigma(\tau)-\pi(\tau) & =\frac{n_1}{p^2n_2(n_1+n_2)}\left(1-\frac{1}{p}-\frac{p-1}{p(p-2)!}\right)^t.
        \end{align*}
        Since \(|B|=p^2n_2\), we have
        \[
            \Vert P^t_\sigma-\pi\Vert_{\mathrm{TV}} =\sum_{\tau \in B}\frac{n_1}{p^2n_2(n_1+n_2)}\left(1-\frac{1}{p}-\frac{p-1}{p(p-2)!}\right)^t=\frac{n_1}{n_1+n_2}\left(1-\frac{1}{p}-\frac{p-1}{p(p-2)!}\right)^t.
        \]
    \end{proof}
\end{appendix}
\end{document}